\documentclass[11pt,a4paper,twoside]{amsart}

\usepackage{amsmath,amssymb,amsthm}
\usepackage{inputenc}
\usepackage{tikz} 
\usepackage{float}
\floatstyle{boxed} 
\usepackage{graphicx,subfigure}
\usepackage{caption}
\usepackage[T1]{fontenc} 
\usepackage{enumitem}
\usepackage{mathtools}
\usepackage{mathrsfs}
\usepackage{xcolor}

\usepackage{stackengine}

\usepackage{hyperref}
\usepackage{cleveref}
\usepackage{autonum}

\usepackage[style=numeric,firstinits=true,sorting=nyt,backend=bibtex,maxnames=99]{biblatex}
\usepackage{etoolbox}
\usepackage{upgreek}
\def\@begintheorem#1#2{\par\bgroup{\sc #1 \ #2. }  \it \\\ignorespace }
\def\@opargbegintheorem#1#2#3{\par\bgroup{\sc #1\ #2 \ (#3).}  \it  \ignorespace}
\def\@endtheorem{\egroup}
\theoremstyle{plain}
\newtheorem{theorem}{Theorem}[section]

\theoremstyle{definition}
\newtheorem{definition}[theorem]{Definition}

\theoremstyle{remark}
\newtheorem{remark}[theorem]{Remark}

\theoremstyle{plain}
\newtheorem{lemma}[theorem]{Lemma}

\theoremstyle{plain}
\newtheorem{corollary}[theorem]{Corollary}

\theoremstyle{plain}
\newtheorem{proposition}[theorem]{Proposition}

\numberwithin{equation}{section}

\newcommand{\C}{{\mathbb C}}

\newcommand{\N}{{\mathbb N}}

\newcommand{\R}{{\mathbb R}}

\newcommand{\ve}{\mathbf{v}_{\!\epsilon}}
\newcommand{\ue}{\mathbf{u}_\epsilon}
\newcommand{\V}{\mathbf{V}}
\newcommand{\Vep}{\mathbf{V}_{\!\epsilon}}
\renewcommand{\S}{{\Sigma}}

\newcommand{\K}{\mathcal{K}}
\newcommand{\Rt}{{\R}^3}
\newcommand{\HH}{\mathcal{H}}

\newcommand{\LL}{\mathcal{L}}
\newcommand{\de}{\partial}

\newcommand{\supp}{{\rm supp}}
\newcommand{\dt}{{\rm dist}}
\newcommand{\diam}{{\rm diam}}
\newcommand{\witb}{\widetilde{B}}
\newcommand{\witm}{\widetilde{M}}
\newcommand{\wita}{\widetilde{A}}
\newcommand{\witt}{\widetilde{T}}

{\left\lbrace\begin{array}{@{}l@{}}}%
{\end{array}\right.}

\newcommand{\sgn}{\mathop{\textrm{sign}}}

\DeclarePairedDelimiter{\seq}{\lbrace}{\rbrace}

\title[Klein's Paradox and the Relativistic $\delta$-shell Interaction in $\Rt$
]{Klein's Paradox and\\the Relativistic $\delta$-shell Interaction in $\Rt$
}
\date{\today}
\author[A. Mas, F. Pizzichillo]{Albert Mas and Fabio Pizzichillo}
\address{A. Mas. Departament de Matem\`atiques,
ETSEIB, Universitat Polit\`ecnica de Catalunya, Avda. Diagonal 647, 08028 Barcelona (Spain)}
\email{amasblesa@gmail.com}
\address{F. Pizzichillo, BCAM - Basque Center for Applied Mathematics,
Alameda de Mazarredo 14, 48009 Bilbao (Spain)}
\email{fpizzichillo@bcamath.org}

\begin{document}
\begin{abstract}
Under certain hypothesis of smallness of the regular potential $\V$, we prove that the Dirac operator in $\Rt$ coupled with a suitable rescaling of  $\V$ converges in the strong resolvent sense to the Hamiltonian coupled with a $\delta$-shell potential supported on $\Sigma$, a bounded $C^2$ surface. Nevertheless, the coupling constant depends non-linearly on the potential $\V$: the Klein's Paradox comes into play.\\

%\noindent{\scshape R\'esum\'e}. Sous une certaine hypoth\`ese de petitesse du potentiel r\'egulier $\V$, nous d\'emontrons que l'op\'erateur de Dirac dans $\Rt$, coupl\'e \`a une mise \`a l'\'echelle convenable de $\V$, converge au sens de la r\'esolvante forte vers le Hamiltonien coupl\'e \`a un potentiel $\delta$ support\'e sur une coque $\Sigma$, une surface $C^2$ born\'ee. N\'eanmoins, la constante de couplage d\'epend non-lin\'eairement du potentiel $\V$: le Paradoxe de Klein intervient.

\end{abstract}
\maketitle
%\thefootnote

\section{Introduction}
The ``Klein's Paradox'' is a counter-intuitive relativistic phenomenon related to scattering theory for high-barrier (or equivalently low-well) potentials for the Dirac equation. When an electron is approaching to a barrier, its wave function can be split in two parts: the reflected one and the transmitted one. In a non-relativistic situation, it is well known that the transmitted wave-function decays exponentially depending on the high of the potential, see \cite{thaller2005advanced} and the references therein. In the case of the Dirac equation it has been observed, in \cite{klein1929reflexion} for the first time, that the transmitted wave-function depends weakly on the power of the barrier, and it becomes almost transparent for very high barriers. This means that outside the barrier the wave-function behaves like an electronic solution and inside the barrier it behaves like a positronic one, violating the principle of the conservation of the charge. This incongruence comes from the fact that, in the Dirac equation, the behaviour of electrons and positrons is described by different components of the same spinor wave-function, see \cite{katsnelson2006chiral}. Roughly speaking, this contradiction derives from the fact that even if a very high barrier is reflective for electrons, it is attractive for the positrons.

From a mathematical perspective, the problem appears when approximating the Dirac operator coupled with a $\delta$-shell potential by  the corresponding operator using local potentials with shrinking support.
The idea of coupling Hamiltonians with singular potentials supported on subsets of lower dimension with respect to the ambient space (commonly called {\em singular perturbations}) is quite classic in quantum mechanics. One important example is the model of a particle in a one-dimensional lattice that analyses the evolution of an electron on a straight line perturbed by a potential caused by ions in the periodic structure of the crystal that create an electromagnetic field. In 1931, Kronig and Penney \cite{kronig1931quantum} idealized this system: in their model the electron is free to move in regions of the whole space separated by some periodical barriers which are zero everywhere except at a single point, where they take infinite value. In a modern language, this corresponds to a $\delta$-point potential.  For the Shr\"oedinger operator, this problem is described in the manuscript \cite{albeverio2012solvable} for finite and infinite $\delta$-point interactions and in \cite{exner2007leaky} for singular potentials supported on hypersurfaces. The reader may look at \cite{ditexnseb, amv1, amv2} and the references therein for the case of the Dirac operator, and to \cite{posilicano1} for a much more general scenario.

Nevertheless, one has to keep in mind that, even if this kind of model is easier to be mathematically understood, since the analysis can be reduced to an algebraic problem, it is and ideal model that cannot be physically reproduced. This is the reason why it is interesting to approximate this kind of operators by more regular ones. For instance, in one dimension, if $V\in C^\infty_c(\R)$ then 
\begin{equation}
V_\epsilon(t):=\textstyle{\frac{1}{\epsilon}\,V\big(\frac{t}{\epsilon}\big)
\to(\int V)}\delta_0\quad\text{when }\epsilon\to0
\end{equation}
in the sense of distributions, where $\delta_0$ denotes the Dirac measure at the origin. 
In \cite{albeverio2012solvable} it is proved that  
$\Delta+V_\epsilon\to\Delta+(\int V)\delta_0$ in the norm resolvent sense when $\epsilon\to0$, and in \cite{approximation} this result is generalized to higher dimensions for singular perturbations on general smooth hypersurfaces.

These kind of results do not hold for the Dirac operator. In fact, in \cite{sebaklein} it is proved that, in the $1$-dimensional case,  the convergence holds in the norm resolvent sense but the coupling constant does depend non-linearly on the potential $V$, unlike in the case of Schr\"oedinger operators. This non-linear phenomenon, which may also occur in higher dimensions, is a consequence of the fact that, in a sense, the free Dirac operator is critical with respect to the set where the $\delta$-shell interaction is performed, unlike the Laplacian (the Dirac/Laplace operator is a first/second order differential operator, respectively, and the set where the interaction is performed has codimension $1$ with respect to the ambient space).
The present paper is devoted to the study of the $3$-dimensional case, where we investigate if it is possible obtain the same results as in one dimension. We advance that, for $\delta$-shell interactions on bounded smooth hypersurfaces, we get the same non-linear phenomenon on the coupling constant but we are only able to show convergence in the strong resolvent sense.

Given $m\geq0$, the free Dirac operator in $\R^3$ is defined by 
\begin{equation}
H:=-i\alpha\cdot\nabla+m\beta, 
\end{equation}
where $\alpha=(\alpha_1,\alpha_2,\alpha_3)$,
\begin{equation}
\alpha_j=\left(\begin{array}{cc}
0& {\sigma}_j\\
{\sigma}_j&0
\end{array}\right)\quad\text{for }j=1,2,3,\quad \beta=\left(\begin{array}{cc}
\mathbb{I}_2&0\\
0&-\mathbb{I}_2
\end{array}\right),\quad \mathbb{I}_2:=\left(
\begin{array}{cc}
1 & 0\\
0 & 1
\end{array}\right),
\end{equation}
\begin{equation}\label{paulimatrices}
\text{and}\quad{\sigma}_1 =\left(
\begin{array}{cc}
0 & 1\\
1 & 0
\end{array}\right),\quad {\sigma}_2=\left(
\begin{array}{cc}
0 & -i\\
i & 0
\end{array}
\right),\quad{\sigma}_3=\left(
\begin{array}{cc}
1 & 0\\
0 & -1
\end{array}\right)
\end{equation}
is the family of \textit{Pauli's matrices}. It is well known that $H$ is self-adjoint on the Sobolev space $H^1(\Rt)^4=:D(H)$, see \cite[Theorem 1.1]{thaller}. Throughout this article we assume that $m>0$.

In the sequel $\Omega\subset\Rt$ denotes a bounded $C^2$ domain and $\Sigma:=\partial\Omega$ denotes its boundary. By a $C^2$ domain we mean the following: for each point $Q\in\S$ there exist
a ball $B\subset\R^3$ centered at $Q$, a $C^2$ function 
$\psi:\R^{2}\to\R$ and a coordinate system $\{(x,x_3):\,x\in\R^{2},\,x_3\in\R\}$ so that, with respect to this coordinate system, $Q=(0,0)$ and
\begin{equation}
\begin{split}
B\cap\Omega=B\cap\{(x,x_3):\,x_3>\psi(x)\},\\
B\cap\S=B\cap\{(x,x_3):\,x_3=\psi(x)\}.
\end{split}
\end{equation}
By compactness, one can find a finite covering of $\S$ made of such coordinate systems, thus the Lipschitz constant of those $\psi$ can be taken uniformly bounded on $\S$.

Set
$\Omega_\epsilon:=\{x \in\Rt : \,d(x,\Sigma)<{\epsilon}\}$ 
for $\epsilon>0$. Following \cite[Appendix B]{approximation}, there exists $\eta>0$ small enough depending on $\S$ so that for every $0<\epsilon\leq\eta$ one can parametrize $\Omega_\epsilon$ as
\begin{equation}\label{C^2 domain properties}
\Omega_\epsilon =\{x_\Sigma+t \nu (x_\Sigma): \,x_\Sigma\in \Sigma,\,t\in(-\epsilon,\epsilon)\},
\end{equation}
where $\nu (x_\Sigma)$ denotes the outward (with respect to $\Omega$) unit normal vector field on $\Sigma$ evaluated at $x_\Sigma$. This parametrization is a bijective correspondence between $\Omega_\epsilon$ and $\S\times(-\epsilon,\epsilon)$, it can be understood as {\em tangential} and {\em normal coordinates}. For $t\in\left[-\eta,\eta\right]$, we set
\begin{equation}\label{C^2 domain properties2}
\Sigma_t:=\{x_\Sigma+t \nu (x_\Sigma): \,x_\Sigma\in \Sigma\}.
\end{equation}
In particular, $\Sigma_t=\partial\Omega_t\setminus\Omega$ if $t>0$, $\Sigma_t=\partial\Omega_{|t|}\cap\Omega$ if $t<0$ and $\Sigma_0=\Sigma$. Let $\upsigma_t$ denote the surface measure on $\Sigma_t$ and, for simplicity of notation, we set $\upsigma:=\upsigma_0$, the surface measure on $\S$.

Given $V\in L^\infty(\R)$ with $\supp V\subset[-\eta,\eta]$ and $0<\epsilon\leq\eta$ define
\begin{equation}
V_\epsilon(t):=\frac{\eta}{\epsilon}\,V\Big(\frac{\eta t}{\epsilon}\Big)
\end{equation}
and, for $x\in\Rt$,
\begin{equation}\label{def bigV}
\Vep(x):=
\begin{cases}
V_\epsilon (t) & \mbox{if } x\in\Omega_\epsilon,\text{ where }x=x_\Sigma+t\nu (x_\Sigma)\text{ for a unique }(x_\Sigma,t)\in\Sigma\times(-\epsilon,\epsilon),\\
0 & \mbox{if } x\not\in\Omega_\epsilon.
\end{cases}
\end{equation}
Finally, set 
\begin{equation}\label{eq u,v}
\begin{split}
\ue:=|\Vep|^{1/2},&\quad
\ve:=\sgn(\Vep)|\Vep|^{1/2},\\
u(t):=|\eta V (\eta t)|^{1/2},&\quad v(t):=\sgn(V(\eta t))u(t).
\end{split}
\end{equation}
Note that $\ue,\ve\in L^\infty(\Rt)$ are supported in $\overline{\Omega_\epsilon}$ and $u,v\in L^\infty(\R)$ are supported in $[-1,1]$.

\begin{definition}\label{deltasmall}
Given $\eta,\,\delta>0$, we say that $V\in L^\infty(\R)$ is $(\delta,\eta)$-small if 
\begin{equation}
\supp V\subset[-\eta,\eta]\quad\text{and}\quad\|V\|_{L^\infty(\R)}\leq\frac{\delta}{\eta}.
\end{equation}
\end{definition}
Observe that if $V$ is $(\delta,\eta)$-small then
$\|V\|_{L^1(\R)}\leq2\delta$, this is the reason why we call it a ``small'' potential. 

In this article we study the asymptotic behaviour, in a strong resolvent sense, of the couplings of the free Dirac operator with electrostatic and Lorentz scalar short-range potentials of the form 
\begin{equation}\label{correc1}
H+\mathbf{V}_{\!\epsilon}\qquad\text{and}\qquad 
H+\beta \mathbf{V}_{\!\epsilon},
\end{equation}
respectively, where ${V}_{\!\epsilon}$ is given by \eqref{def bigV} for some $(\delta,\eta)$-small $V$ with $\delta$ and $\eta$ small enough  only depending on $\S$.
By \cite[Theorem 4.2]{thaller}, both couplings in \eqref{correc1} are self-adjoint operators on $H^1(\Rt)^4$. 
Given $\eta>0$ small enough so that \eqref{C^2 domain properties} holds, and given $u$ and $v$ as in \eqref{eq u,v} for some $V\in{L^\infty(\R)}$ with $\supp V\subset[-\eta,\eta]$, set
\begin{equation}\label{correc3}
\K_V f(t):=\frac{i}{2}\int_\R u(t)\sgn(t-s)v(s)f(s)\,ds
\quad\text{ for $f\in L^1_{loc}(\R)$}.
\end{equation}
The main result in this article reads as follows.
\begin{theorem}\label{Main theorem}
There exist $\eta_0,\,\delta>0$ small enough only depending on $\S$ such that, for any $0<\eta\leq\eta_0$ and $(\delta,\eta)$-small $V$,
\begin{align}\label{main eq}
&H+\mathbf{V}_{\!\epsilon}\to H+\lambda_e\delta_\Sigma\quad\text{in the strong resolvent sense when $\epsilon\to0$},\\\label{main eq 2}
&H+\beta\mathbf{V}_{\!\epsilon}\to H+\lambda_s\beta\,\delta_\Sigma\quad\text{in the strong resolvent sense when $\epsilon\to0$},
\end{align} 
where 
\begin{align}\label{def lambda elec}
\lambda_e &:=\textstyle{\int_\R}v(t)\,((1-\K_V^2)^{-1}u)(t)\,dt\in\R, \\\label{def lambda scalar}
\lambda_s&:=\textstyle{\int_\R}v(t)\,((1+\K_V^2)^{-1}u)(t)\,dt\in\R
\end{align}
and $H+\lambda_e\delta_\Sigma$ and $H+\lambda_s\beta\,\delta_\Sigma$ are the electrostatic and Lorentz scalar shell interactions given by \eqref{eq defi electro} and \eqref{eq defi scalar}, respectively.  
\end{theorem}

To define $\lambda_e$ in \eqref{def lambda elec} and $\lambda_s$ in \eqref{def lambda scalar}, the invertibility of $1\pm\K_V^2$ is required. However, since $\K_V$ is a Hilbert-Schmidt operator, we know that 
$\|\K_V\|_{L^2(\R)\to L^2(\R)}$ is controlled by the norm of its kernel in $L^2(\R\times\R)$, which is exactly 
$\|u\|_{L^2(\R)}\|v\|_{L^2(\R)}=\|V\|_{L^1(\R)}\leq 2\delta<1$, assuming that $\delta<1/2$ and that $V$ is $(\delta,\eta)$-small with $\eta\leq\eta_0$. We must stress that the way to construct $\lambda_e$ and $\lambda_s$ is the same as in  the one dimensional case, see \cite[Theorem 1]{sebaklein}.

From \Cref{Main theorem} we deduce that if $a\in\sigma(H+\lambda_e\delta_{\Sigma})$, where $\sigma(\cdot)$ denotes the spectrum, then there exists a sequence $\seq{a_\epsilon}$ such that $a_\epsilon\in \sigma(H+\mathbf{V}_{\!\epsilon})$ and $a_\epsilon\to a$ when $\epsilon\to0$.  Contrary to  what happens if norm resolvent convergence holds, the vice-versa spectral implication may not hold. That is, if $a_\epsilon \to a$ with $a_\epsilon\in \sigma(H+\mathbf{V}_{\!\epsilon})$, it may occur that $a\notin\sigma(H+\lambda_e\delta_{\Sigma})$. The same happens for the Lorentz scalar case.
We should highlight that the kind of instruments we used to prove \Cref{Main theorem} suggest us that the norm resolvent convergence may not hold in general. Nevertheless, if $\Sigma$ is a sphere, the vice-versa spectral implication does hold. That means that, passing to the limit, we don't lose any element of the spectrum for electrostatic and scalar spherical $\delta$-shell interactions, see \cite{sphericalnotes}.

The non-linear behaviour of the limiting coupling constant with respect to the approximating potentials mentioned in the first paragraphs of the introduction is depicted by \eqref{def lambda elec} and \eqref{def lambda scalar}; the reader may compare this to the analogous result \cite[Theorem 1.1]{approximation} in the non-relativistic scenario.
However, unlike in \cite[Theorem 1.1]{approximation}, in Theorem \ref{Main theorem} we demand an smallness assumption on the potential, the $(\delta,\eta)$-smallness from Definition \ref{deltasmall}. 
We use this assumption in Corollary \ref{convergence main} below, where the strong convergence of some inverse operators $(1+B_\epsilon(a))^{-1}$ when $\epsilon\to0$ is shown. The proof of Theorem \ref{Main theorem} follows the strategy of \cite[Theorem 1.1]{approximation}, but dealing with the Dirac operator instead of the Laplacian makes a big difference at this point. In the non-relativistic scenario, the fundamental solution of $-\Delta+a^2$ in $\Rt$ for $a>0$ has exponential decay at infinity and behaves like $1/|x|$ near the origin, which is locally integrable in $\R^2$ and thus its integral tends to zero as we integrate on shrinking balls in $\R^2$ centered at the origin. This facts are used in \cite{approximation} to show that their corresponding $(1+B_\epsilon(a))^{-1}$ can be uniformly bounded in $\epsilon$ just by taking $a$ big enough. In our situation, the fundamental solution of 
$H-a$ in $\Rt$ can still be taken with exponential decay at infinity for $a\in\C\setminus\R$, but it is not locally absolutely integrable in $\R^2$. Actually, its most singular part behaves like $x/|x|^3$ near the origin, and thus it yields a singular integral operator in $\R^2$. This means that the contribution near the origin can not be disesteemed as in \cite{approximation} just by shrinking the domain of integration and taking $a\in\C\setminus\R$ big enough, something else is required. We impose smallness on $V$ to obtain smallness on $B_\epsilon(a)$ and ensure the uniform invertibility of $1+B_\epsilon(a)$ with respect to $\epsilon$; this is the only point where the $(\delta,\eta)$-smallenss is used. 

Let $\eta_0,\,\delta>0$ be as in Theorem \ref{Main theorem}. Take $0<\eta\leq\eta_0$ and $V=\frac{\tau}{2} \chi_{(-\eta,\eta)}$ for some $\tau\in\R$ such that $0<|\tau|\eta\leq2\delta$. Then, arguing as in \cite[Remark 1]{sebaklein}, one gets that  
\begin{equation}
\int_\R\!v\,(1-\K_V^2)^{-1}u=\sum_{n=0}^\infty\int_\R\!v\, \K_V^{2n}u=2\tan\Big(\frac{\tau\eta}{2}\Big).
\end{equation}
Since $V$ is $(\delta,\eta)-$small,
using \eqref{def lambda elec} and \eqref{main eq} we obtain that \begin{equation}H+\mathbf{V}_{\!\epsilon}\to H+2\tan(\textstyle{\frac{\tau\eta}{2}})\delta_\Sigma\quad\text{ in the strong resolvent sense when $\epsilon\to0$,}
\end{equation}
analogously to {\cite[Remark 1]{sebaklein}}.
Similarly, one can check that $\int\!v\,(1+\K_V^2)^{-1}u=2\tanh(\textstyle{\frac{\tau\eta}{2}})$. 
Then, \eqref{def lambda scalar} and \eqref{main eq 2} yield 
\begin{equation}H+\beta\,\mathbf{V}_{\!\epsilon}\to H+2\tanh(\textstyle{\frac{\tau\eta}{2}})\beta\delta_\Sigma\quad\text{ in the strong resolvent sense when $\epsilon\to0$.}
\end{equation} 

Regarding the structure of the paper, Section \ref{s preli} is devoted to the preliminaries, which refer to basic rudiments with a geometric measure theory flavour and spectral properties of the short range and shell interactions appearing in Theorem \ref{Main theorem}. 
In Section \ref{s main deco} we present the first main step to prove Theorem \ref{Main theorem}, a decomposition of the resolvent of the approximating interaction into three concrete operators. This type of decomposition, which is made through a scaling operator, already appears in \cite{approximation, sebaklein}. Section \ref{s main deco} also contains some auxiliary results concerning these three operators, whose proofs are carried out later on, and the proof of Theorem \ref{Main theorem}, see Section \ref{s2 ss1}. Sections \ref{ss C}, \ref{ss B}, \ref{ss A} and \ref{s proof corol}
are devoted to prove all those auxiliary results presented in Section \ref{s main deco}.
\section*{Acknowledgement}
We would like to thank Luis Vega for the enlightening discussions. Both authors were partially supported by the ERC Advanced Grant 669689 HADE (European Research Council).
Mas was also supported by the {\em Juan de la Cierva} program JCI2012-14073 and the project MTM2014-52402 (MINECO, Gobierno de Espa\~na).
Pizzichillo was also supported by the MINECO project MTM2014-53145-P, by the Basque Government through the BERC 2014-2017 program and by the Spanish Ministry of Economy and Competitiveness MINECO: BCAM Severo Ochoa accreditation SEV-2013-0323.

\section{Preliminaries}\label{s preli}
As usual, in the sequel the letter `$C$' (or `$c$') stands
for some constant which may change its value at different
occurrences. We will also make use of constants with subscripts, both to highlight the dependence on some other parameters and to stress that they retain their value from one equation to another. The precise meaning of the subscripts will be clear from the context in each situation.

\subsection{Geometric and measure theoretic considerations}\label{s1 ss1}
\mbox{}

In this section we recall some  geometric and measure theoretic properties of $\Sigma$ and the domains presented in \eqref{C^2 domain properties}. At the end, we provide some growth estimates of the measures associated to the layers introduced in \eqref{C^2 domain properties2}.

The following definition and propositions correspond to Definition 2.2 and Propositions 2.4 and 2.6 in \cite{approximation}, respectively. The reader should look at \cite{approximation} for the details.

\begin{definition}[Weingarten map]\label{defi weingarten}
Let $\Sigma$ be parametrized by the family $\{\varphi_i,U_i,V_i\}_{i\in I}$, that is, $I$ is a finite set, $U_i\subset\R^2$, $V_i\subset\Rt$, $\S\subset\cup_{i\in I}V_i$ and $\varphi_i(U_i)=V_i\cap\S$ for all $i\in I$. For \[x=\varphi_i(u)\in \Sigma\cap V_i\] with $u\in U_i$, $i\in I$, one defines the Weingarten map $W(x): T_x\to T_x$, where $T_x$ denotes the tangent space of $\S$ on $x$, as the linear operator acting on the basis vector $\{\de_j\varphi_i(u)\}_{j=1,2}$ of $T_x$ as
\[
W(x)\de_j\varphi_i(u):=-\de_j\nu(\varphi_i(u)).
\]
\end{definition}
 
\begin{proposition}\label{weingarten map}
The Weingarten map $W(x)$ is symmetric with respect to the inner product induced by the first fundamental form and its eigenvalues are uniformly bounded for all $x\in\Sigma$.
\end{proposition}

Given $0<\epsilon\leq\eta$ and $\Omega_\epsilon$ as in \eqref{C^2 domain properties}, let $i_\epsilon: \Sigma\times(-\epsilon,\epsilon)\to \Omega_\epsilon$ be the bijection defined by 
\begin{equation}\label{i epsilon}
i_\epsilon(x_\Sigma,t):=x_\Sigma+t \nu (x_\Sigma).
\end{equation}
For future purposes, we also introduce the projection $P_\Sigma:\Omega_\epsilon \to \Sigma$ given by 
\begin{equation}\label{P Sigma}
P_\Sigma (x_\S+t\nu(x_\S)):=x_\S.
\end{equation}

For $1\leq p<+\infty$, let $L^p(\Omega_\epsilon)$ and $L^p(\Sigma\times(-1,1))$ be the Banach spaces endowed with the norms
\begin{equation}\label{eqn:coaeraqqq}
\|f\|_{L^p(\Omega_\epsilon)}^p
:=\int_{\Omega_\epsilon}|f|^p\,d\LL,\qquad
\|f\|_{L^p(\Sigma\times(-1,1))}^p
:=\int_{-1}^1\int_{\Sigma}|f|^p\,d\upsigma\,dt,
\end{equation}
respectively, where $\LL$ denotes the Lebesgue measure in $\R^3$. The Banach spaces corresponding to the endpoint case $p=+\infty$ are defined, as usual, in terms of essential suprema with respect to the measures associated to $\Omega_\epsilon$ and $\Sigma\times(-1,1)$ in \eqref{eqn:coaeraqqq}, respectively.

\begin{proposition}\label{prop:coarea}
If $\eta>0$ is small enough, there exist $0<c_1,c_2<+\infty$ such that
\[
c_1\|f\|_{L^1(\Omega_\epsilon)}\leq\|f \circ i_\epsilon\|_{L^1(\Sigma\times (-\epsilon,\epsilon))}\leq c_2\|f\|_{L^1(\Omega_\epsilon)}\quad\text{for all }f \in L^1(\Omega_\epsilon),\, 0<\epsilon\leq\eta.
\]
Moreover, if $W$ denotes the Weingarten map associated to $\Sigma$ from {\em Definition \ref{defi weingarten}},
\begin{equation}\label{eqn:coaera}
\int_{\Omega_\epsilon}f(x)\,dx=\int_{-\epsilon}^\epsilon\int_{\Sigma} f(x_\Sigma+t\nu(x_\Sigma))\det(1-tW(x_\Sigma))\,d\upsigma(x_\Sigma)\,dt\quad\text{for all }f \in L^1(\Omega_\epsilon).
\end{equation}
\end{proposition}
The eigenvalues of the Weingarten map $W(x)$ are the principal curvatures of $\S$ on $x\in\S$, and they are independent of the parametrization of $\S$. Therefore, the term $\det(1-tW(x_\Sigma))$ in \eqref{eqn:coaera} is also independent of the parametrization of $\S$.

\begin{remark}
Let $h:\Omega_\epsilon\to(-\epsilon,\epsilon)$ be defined by 
$h(x_\Sigma+t\nu(x_\Sigma)):=t$. Then $|\nabla h|=1$ in $\Omega_\epsilon$, so the coarea formula (see \cite[Remark 2.94]{ambrosiofuscopallara}, for example) gives
\begin{equation}\label{eqn:coaera3}
\int_{\Omega_\epsilon}f(x)\,dx=\int_{-\epsilon}^\epsilon\int_{\Sigma_t} f(x)\,d\upsigma_t(x)\,dt\quad\text{for all }f \in L^1(\Omega_\epsilon).
\end{equation}
In view of \eqref{eqn:coaera}, one deduces that
\begin{equation}\label{eqn:coaera2}
\int_{\Sigma_t} f\,d\upsigma_t
=\int_{\Sigma}f(x_\Sigma+t\nu(x_\Sigma))\det(1-tW(x_\Sigma))\,d\upsigma(x_\Sigma)
\end{equation}
for all $t\in(-\epsilon,\epsilon)$ and all $f \in L^1(\Sigma_t)$.
\end{remark}

In the following lemma we give uniform growth estimates on the measures $\upsigma_t$, for $t\in[-\eta,\eta]$, that exhibit their 2-dimensional nature. These estimates will be used many times in the sequel, mostly for the case of $\upsigma$.
\begin{lemma}\label{2d AD regularity}
If $\eta>0$ is small enough, there exist $c_1,c_2>0$ such that
\begin{eqnarray}\label{sigma_t in Delta}
& &\upsigma_t(B_r(x))\leq c_1 r^2\quad \text{for all }x\in \Rt,\, r>0,\, t\in[-\eta,\eta],\label{sigma_t in Sigma2}\\
& &\upsigma_t(B_r(x))\geq c_2{r^2}\quad \text{for all }x\in \Sigma_t,\,0<r<2\diam(\Omega_\eta),\,t\in[-\eta,\eta],\label{sigma_t in Sigma}
\end{eqnarray}
being $B_r(x)$ the ball of radius $r$ centred at $x$.
\end{lemma}
\begin{proof}
We first prove \eqref{sigma_t in Delta}. Let $r_0>0$ be a constant small enough to be fixed later on. 
If $r\geq r_0$, then 
\[
\upsigma_t(B_r(x))\leq \max_{t\in[-\eta,\eta]}\upsigma_t(\Rt)\leq C=\frac{C}{r_0^2}\,r_0^2\leq C_0r^2,
\]
where $C_0:=C/{r_0^2}>0$ only depends on $r_0$ and $\eta$. Therefore, we can assume that $r< r_0$. 
Let us see that we can also suppose that $x\in\Sigma_t$. 
In fact, if $\eta$ and $r_0$ are small enough and $0<r<r_0$, given $x\in\Rt$ one can always find $\tilde{x}\in \Sigma_t$ such that $\upsigma_t(B_r(x))\leq 2\upsigma_t(B_{{r}}(\tilde{x}))$ (if $x\in\Omega_\eta$ just take $\tilde{x}=P_\S x+t\nu(P_\S x)$).
Then if \eqref{sigma_t in Delta} holds for $\tilde{x}$, one gets
$\upsigma_t(B_r(x))\leq2\upsigma_t(B_{{r}}(\tilde{x}))\leq C r^2,$ as desired.

Thus, it is enough to prove \eqref{sigma_t in Delta} for $x\in \Sigma_t$ and $r<r_0$. If $r_0$ and $\eta$ are small enough, covering $\S_t$ by local chards we can find an open and bounded set $V_{t,r}\subset \R^2$ and a $C^1$ diffeomorphism $\varphi_t:\R^2\to \varphi_t(\R^2)\subset\Rt$  such that $\varphi_t(V_{t,r})=\Sigma_t\cap B_r(x)$. By means of a rotation if necessary, we can further assume that $\varphi_t$ is of the form $\varphi_t(y')=(y',T_t(y'))$, i.e. $\varphi_t$ is the graph of a $C^1$ function $T_t:\R^2\to \R$, and that
$\max_{t\in [-\eta,\eta]}\|\nabla T_t\|_\infty\leq C$ (this follows from the regularuty of $\S$).
Then, if $x'\in V_{t,r}$ is such that $\varphi_t(x')=x$, for any $y'\in V_{t,r}$ we get 
\[
r^2\geq|\varphi_t(y')-\varphi_t(x')|^2\geq |y'-x'|^2,
\]
which means that $V_{t,r}\subset\{y'\in\R^2:\,|x'-y'|<r\}=:B'\subset\R^2$.
Denoting by $\mathcal{H}^2$ the 2-dimensional Hausdorff measure, from \cite[Theorem 7.5]{mattila} we get
\[
\upsigma_t(B_r(x))=\mathcal{H}^2(\varphi_t(V_{t,r}))\leq \mathcal{H}^2(\varphi_t(B')) \leq \|\nabla\varphi_t\|_\infty^2 \mathcal{H}^2(B')\leq C r^2
\]
for all $t\in[-\eta,\eta]$,
so \eqref{sigma_t in Delta} is finally proved. 

Let us now deal with \eqref{sigma_t in Sigma}. Given $r_0>0$, by the regularity and boundedness of $\Sigma$ it is clear that
$\inf_{t\in[-\eta,\eta],\,x\in\Sigma_t}\upsigma_t(B_{r_0}(x))\geq C>0$. As before, for any $r_0\leq r<2\diam(\Omega_\eta)$ we easily see that
\begin{equation}\upsigma_t(B_r(x))\geq \upsigma_t(B_{r_0}(x))\geq C=\frac{C}{4\diam(\Omega_\eta)^2}\,4\diam(\Omega_\eta)^2\geq C_1r^2,\end{equation}
where $C_1:={C}/{4\diam(\Omega_\eta)^2}>0$ only depends on $r_0$ and $\eta$. Hence \eqref{sigma_t in Sigma} is proved for all $r_0\leq r<2\diam(\Omega_\eta)$. 

The case $0<r<r_0$ is treated, as before, using the local parametrization of $\S_t$ around $x$ by the graph of a function. Taking $\eta$ and $r_0$ small enough, we may assume the existence of $V_{t,r}$ and $\varphi_t$ as above, so let us set $\varphi_t(x')=x$ for some $x'\in V_{t,r}$. The fact that $\varphi_t$ is of the form $\varphi_t(y')=(y',T_t(y'))$ and that $\varphi_t(V_{t,r})=\Sigma_t\cap B_r(x)$ implies that $B'':=\{y'\in\R^2:\,|x'-y'|<C_2r\}\subset V_{t,r}$ for some $C_2>0$ small enough only depending on $\max_{t\in [-\eta,\eta]}\|\nabla T_t\|_\infty$, which is finite by assumption. Then, we easily see that
\begin{equation}\upsigma_t(B_r(x))=\upsigma_t(\varphi_t(V_{t,r}))
\geq \upsigma_t(\varphi_t(B''))=\int_{B''}\sqrt{1+|\nabla T_t(y')|^2}\,dy'
\geq\int_{B''}dy'=Cr^2,\end{equation}
where $C>0$ only depends on $C_2$. The lemma is finally proved.
\end{proof}

\subsection{Shell interactions for Dirac operators}\label{section shell interactions}\label{s1 ss2}
\mbox{}

In this section we briefly recall some useful instruments regarding the  $\delta$-shell interactions studied in \cite{amv1,amv2}. The reader should look at \cite[Section 2 and Section 5]{amv2} for the details.

Let $a\in \C$. A fundamental solution of $H-a$ is given by
\begin{equation}
\phi^a (x)=\frac{e^{-\sqrt{m^2-a^2}|x|}}{4\pi|x|}\Big(a+m\beta +\Big(1+\sqrt{m^2-a^2}|x|\Big)\,i\alpha\cdot\frac{x}{|x|^2}\Big)\quad \text{for }x\in\Rt\setminus\{0\},
\end{equation}
where $\sqrt{m^2-a^2}$ is chosen with positive real part whenever $a\in(\C\setminus\R)\cup\big((-m,m)\times\{0\}\big)$. To guarantee the exponential decay of $\phi^a$ at $\infty$, from now on we assume that $a\in(\C\setminus\R)\cup\big((-m,m)\times\{0\}\big)$.
Given $G\in L^2(\Rt)^4$ and $g\in L^2(\upsigma)^4$ we define
\begin{equation}\label{defi Phia}
\Phi^a(G,g)(x):=\int_{\Rt}\phi^a (x-y)\, G(y)\,dy
+\int_{\Sigma}\phi^a (x-y)g(y)\,d\upsigma(y)
\quad\text{for }x\in\Rt\setminus\Sigma.
\end{equation}
Then, $\Phi^a:L^2(\Rt)^4\times L^2(\upsigma)^4\to L^2(\Rt)^4$ is linear and bounded and $\Phi^a(G,0)\in H^1(\Rt)^4$. 
We also set 
\begin{equation}
\Phi^a_\upsigma G:=\operatorname{tr}_{\upsigma}(\Phi^a(G,0))\in L^2(\upsigma)^4,
\end{equation}
being $\operatorname{tr}_{\upsigma}$ the trace operator on $\Sigma$.
Finally, given $x\in\Sigma$ we define
\begin{equation}
C_\upsigma^a g(x)
:=\lim_{\epsilon\searrow 0}\int_{\Sigma\cap\{|x-y|>\epsilon\}}\phi^a (x-y) g(y)\,d\upsigma(y)
\quad\text{and}\quad C^a_{\pm}g(x):=\lim_{\Omega_\pm\ni y \overset{nt}{\to}x}\Phi^a(0,g)(y),
\end{equation}
where $\Omega_\pm\ni y \overset{nt}{\to}x$ means that $y$ tends to $x$ non-tangentially from the interior/exterior of $\Omega$, respectively, i.e. $\Omega_+:=\Omega$ and $\Omega_-:=\Rt\setminus\overline{\Omega}$.  The operators $C_\upsigma^a$ and $C^a_\pm$ are linear and bounded in $L^2(\upsigma)^4$. Moreover, the following Plemelj-Sokhotski jump formulae holds:
\begin{equation}\label{Plemelj jump formulae}
C^a_\pm=\mp \frac{i}{2}(\alpha\cdot\nu)+C^a_\upsigma.
\end{equation}

Let $\lambda_e\in\R$. Using $\Phi^a$, we define the electrostatic $\delta$-shell interaction appearing in Therorem \ref{Main theorem}  as follows:
\begin{equation}\label{eq defi electro}
\begin{split}
&D(H+\lambda_e\delta_\Sigma):=\{\Phi^0(G,g): \,G\in L^2(\Rt)^4,\, g \in L^2(\upsigma)^4,\, \lambda_e\Phi^0_\upsigma G=-(1+\lambda_e C^0_\upsigma)g\},\\
&(H+\lambda_e\delta_\Sigma)\varphi:=H\varphi+\lambda_e\frac{\varphi_++\varphi_-}{2}\,\upsigma\quad
\text{for }\varphi\in D(H+\lambda_e\delta_\Sigma),
\end{split}
\end{equation}
where $H\varphi$ in the right hand side of the second statement in \eqref{eq defi electro} is understood in the sense of distributions and $\varphi_\pm$ denotes the boundary traces of $\varphi$ when one approaches to $\Sigma$ from $\Omega_\pm$. In particular, one has $(H+\lambda_e\delta_\Sigma)\varphi=G\in L^2(\Rt)^4$ for all $\varphi=\Phi^0(G,g)\in D(H+\lambda_e\delta_\Sigma)$.
We should mention that one recovers the free Dirac operator in $H^1(\Rt)^4$ when $\lambda_e=0$.

From \cite[Section 3.1]{amv2} we know that $H+\lambda_e\delta_\Sigma$ is self-adjoint for all $\lambda_e\neq \pm 2$. Besides, if $\lambda_e\neq0$,
given $a\in(-m,m)$ and $\varphi=\Phi^0(G,g)\in D(H+\lambda_e\delta_\Sigma)$, 
\begin{equation}\label{BS principle}
(H+\lambda_e\delta_\Sigma-a)\varphi=0\quad\text{if and only if}\quad(\textstyle{\frac{1}{\lambda_e}}+C^a_\upsigma)g=0.
\end{equation}
This corresponds to the Birman-Swinger principle in the electrostatic $\delta$-shell interaction setting. Since the case $\lambda_e=0$ corresponds to the free Dirac operator, it can be excluded from this consideration because it is well known that the free Dirac operator doesn't have pure point spectrum. Moreover, the relation \eqref{BS principle} can be easily extended to the case of 
$a\in(\C\setminus\R)\cup\big((-m,m)\times\{0\}\big)$ (one still has exponential decay of a fundamental solution of $H-a$).

In the same vein, given $\lambda_s\in\R$, we define the Lorentz scalar $\delta$-shell interaction as follows:
\begin{equation}\label{eq defi scalar}
\begin{split}
&D(H+\lambda_s\beta\,\delta_\Sigma):=\{\Phi^0(G,g): \,G\in L^2(\Rt)^4,\, g \in L^2(\upsigma)^4,\, \lambda_s\Phi^0_\upsigma G=-(\beta+\lambda_s C^0_\upsigma)g\},\\
&(H+\lambda_s\beta\,\delta_\Sigma)\varphi
:=H\varphi+\lambda_s\beta\,\frac{\varphi_++\varphi_-}{2}\,\upsigma\quad
\text{for }\varphi\in D(H+\lambda_s\beta\,\delta_\Sigma).
\end{split}
\end{equation}
From \cite[Section 5.1]{amv2} we know that $H+\lambda_s\beta\,\delta_\Sigma$ is self-adjoint for  all $\lambda_s\in \R $. Besides, 
given $\lambda_s\neq0$, $a\in(\C\setminus\R)\cup\big((-m,m)\times\{0\}\big)$ and $\varphi=\Phi^0(G,g)\in D(H+\lambda_s\beta\,\delta_\Sigma)$, arguing as in \eqref{BS principle} one gets
\begin{equation}\label{BS principle scalar}
(H+\lambda_s\beta\,\delta_\Sigma-a)\varphi=0\quad\text{if and only if}\quad(\textstyle{\frac{\beta}{\lambda_s}}+C^a_\upsigma)g=0.
\end{equation}

The following lemma describes the resolvent operator of the $\delta$-shell interactions presented in \eqref{eq defi electro} and \eqref{eq defi scalar}.
\begin{lemma}
Given $\lambda_e,\,\lambda_s\in\R$ with $\lambda_e\neq \pm 2$, $a\in\C\setminus\R$ and $F\in L^2(\Rt)^4$, the following identities hold:
\begin{align}\label{resolvent H+lambda delta}
&(H+\lambda_e\delta_\Sigma-a)^{-1}F=(H-a)^{-1}F-\lambda_e\Phi^a\big(0,\left(1+\lambda_e C^a_\upsigma\right)^{-1}\Phi^a_\upsigma F\big),\\\label{resolvent H+lambda beta delta}
&(H+\lambda_s\beta\,\delta_\Sigma-a)^{-1}F=(H-a)^{-1}F-\lambda_s\Phi^a\big(0,\left(\beta+\lambda_s C^a_\upsigma\right)^{-1}\Phi^a_\upsigma F\big).
\end{align}
\end{lemma}
\begin{proof}
We will only show \eqref{resolvent H+lambda delta}, the proof of \eqref{resolvent H+lambda beta delta} is analogous.
Since $H+\lambda_e\delta_\Sigma$ is self-adjoint for $\lambda_e\neq \pm 2$, $(H+\lambda_e\delta_\Sigma-a)^{-1}$ is well-defined and bounded in $L^2(\Rt)^4$. For $\lambda_e=0$ there is nothing to prove, so we assume $\lambda_e\neq 0$.

Let $\varphi=\Phi^0(G,g)\in D(H+\lambda_e\delta_\Sigma)$ as in \eqref{eq defi electro} and $F=(H+\lambda_e\delta_\Sigma-a)\varphi\in L^2(\Rt)^4$. Then,
\begin{equation}\label{F-G=a Phi(G+g)}
F=(H+\lambda_e\delta_\Sigma-a)\Phi^0(G,g)=G-a\Phi^0(G,g).
\end{equation}
If we apply $H$ on both sides of \eqref{F-G=a Phi(G+g)} and we use that $H\Phi^0(G,g)=G+g\upsigma$ in the sense of distributions, we get $HF=HG-a(G+g\upsigma)$, that is, $(H-a)G=(H-a)F+aF+ag\upsigma$. 
Convolving with $\phi^a$ the left and right hand sides of this last equation, we obtain
$G=F+a\Phi^a(F,0)+a\Phi^a(0,g)$, thus $G-F=a\Phi^a(F,g)$. This, combined with \eqref{F-G=a Phi(G+g)}, yields
\begin{equation}\label{Phi0(G,g)=Phia(F,g)}
\Phi^0(G,g)=\Phi^a(F,g).
\end{equation}
Therefore, taking non-tangential boundary values on $\Sigma$ from inside/outside of $\Omega$ in \eqref{Phi0(G,g)=Phia(F,g)} we obtain  \begin{equation}
\Phi^0_\upsigma G+C^0_\pm g=\Phi^a_\upsigma F +C^a_\pm g.\end{equation}  
Since $\Phi^0(G,g)\in D(H+\lambda_e\delta_\Sigma)$, thanks to \eqref{eq defi electro} and \eqref{Plemelj jump formulae} we conclude that
\begin{equation}\label{resolvent eq1}
\Phi_\upsigma^a F=-\Big(\frac{1}{\lambda_e}+C^a_\upsigma\Big) g.
\end{equation}

Since $a\in\C\setminus\R$ and $H+\lambda_e\delta_\Sigma$ is self-adjoint for $\lambda_e\neq\pm2$, by \eqref{BS principle} we see that $\text{Kernel}(\frac{1}{\lambda_e}+C^a_\upsigma)=\{0\}$. Moreover, using the ideas of the proof of \cite[Lemma 3.7]{amv1} and that $\lambda_e\neq \pm 2$, one can show that 
$\frac{1}{\lambda_e}+C^a_\upsigma$ has closed range.
Finally, since we are taking the square root so that 
\begin{equation}
\overline{\sqrt{m^2-a^2}}=\sqrt{m^2-\bar{a}^2},
\end{equation} 
following \cite[Lemma 3.1]{amv1} we see that $\overline{(\phi^a)^t}(x)=\phi^{\bar{a}}(-x)$. Here, $(\phi^a)^t$ denotes the transpose matrix of $\phi^a$. Thus we conclude that $(\text{Range}(\frac{1}{\lambda_e}+C^a_\upsigma))^{\perp}=\text{Kernel}(\frac{1}{\lambda_e}+C^{\bar{a}}_\upsigma)=\{0\}$, and so $\frac{1}{\lambda_e}+C^a_\upsigma$ is invertible. Then, by \eqref{resolvent eq1}, we obtain 
\begin{equation}\label{resolvent eq2}
g=-\Big(\frac{1}{\lambda_e}+C^a_\upsigma\Big)^{-1}\Phi_\upsigma^a F.
\end{equation}
Thanks to \eqref{Phi0(G,g)=Phia(F,g)} and \eqref{resolvent eq2}, we finally get
\begin{align}
(H+\lambda_e\delta_\Sigma-a)^{-1}F&=\varphi=\Phi^0(G,g)=\Phi^a(F,g)
=\Phi^a\Big( F,-\Big(\frac{1}{\lambda_e}+C^a_\upsigma\Big)^{-1}\Phi_\upsigma^a F\Big)\\
&=\Phi^a(F,0)-\lambda_e\Phi^a\big(0,\left(1+\lambda_e C^a_\upsigma\right)^{-1}\Phi^a_\upsigma F\big),
\end{align}
and the lemma follows because $\Phi^a(\cdot,0)=(H-a)^{-1}$ as a bounded operator in $L^2(\Rt)^4$.
\end{proof}

\subsection{Coupling the free Dirac operator with short range potentials as in \eqref{correc1}}\label{ss coupling Ve}
\mbox{}

Given $\Vep$ as in \eqref{def bigV}, set
\begin{equation}
H^e_\epsilon:=H+\Vep\qquad\text{and}\qquad 
H^s_\epsilon:=H+\beta \Vep.
\end{equation} 
Recall that these operators are self-adjoint on $H^1(\Rt)^4$.  In the following, we give the resolvent formulae for $H^e_\epsilon$ and $H^s_\epsilon$.

Throughout this section we make an abuse of notation. Remember that, given $G\in L^2(\Rt)^4$ and $g\in L^2(\upsigma)^4$, in \eqref{defi Phia} we already defined $\Phi^a(G,g)$. However, now we make the identification  $\Phi^a(\cdot)\equiv\Phi^a(\cdot,0)$, that is, in this section we identify $\Phi^a$ with an operator acting on $L^2(\Rt)^4$ by always assuming that the second entrance in $\Phi^a$ vanishes. Besides, in this section we use the symbol $\sigma(\cdot)$ to denote the spectrum of an operator, the reader sholud not confuse it with the symbol $\upsigma$ for the surface measure on $\S$.

\begin{proposition}\label{propo 28}
Let $\ue$ and $\ve$ be as in  \eqref{eq u,v}. Then,
\begin{enumerate}[label=$(\roman*)$]
\item $a\in\rho(H^e_\epsilon)$ if and only if $-1\in\rho(\ue\Phi^a\ve)$, where $\rho(\cdot)$ denotes the resolvent set,
\item $a\in\sigma_{pp}(H^e_\epsilon)$ if and only if $-1\in \sigma_{pp}(\ue\Phi^a\ve)$, where $\sigma_{pp}(\cdot)$ denotes the pure point spectrum. Moreover, the multiplicity of $a$ as eigenvalue of $H^e_\epsilon$ coincides with the multiplicity of $-1$ as eigenvalue of $\ue\Phi^a\ve$.
\end{enumerate} 
Furthermore, the following resolvent formula holds:
\begin{equation}\label{Birman Shwinger}
(H^e_\epsilon-a)^{-1}=\Phi^a - \Phi^a \ve\left(1+\ue\Phi^a\ve\right)^{-1}\ue \Phi^a.
\end{equation}
\end{proposition}
\begin{proof}
To prove $(i)$ and $(ii)$ it is enough to verify that the assumptions of  \cite[Lemma 1]{konnokuroda} are satisfied. That is, we just need to show that $a\in\sigma_{pp}(H^e_\epsilon)$ if and only if $-1\in \sigma_{pp}(\ue\Phi^a\ve)$ and that there exists $a \in \rho(H^e_\epsilon)$ such that $-1\in \rho(\ue\Phi^a\ve)$.

Assume that $a\in\sigma_{pp}(H^e_\epsilon)$. Then $(H+\Vep-a)F=0$ for some $F\in L^2(\Rt)^4$ with $F\not\equiv 0$, so $(H-a)F=-\Vep F$. Using that $\sigma(H)=\sigma_{ess}(H)$, where $\sigma_{ess}(\cdot)$ denotes the essential spectrum, it is not hard to show that indeed $\Vep F\not\equiv0$.  Since $\Vep=\ve \ue$, by setting $G=\ue F\in L^2(\Rt)^4$ we get that $G\not\equiv0$ and 
\begin{equation}\label{(H-a)F=ve G}
(H-a)F=-\ve G.
\end{equation} 
From \cite[Theorem 4.7]{thaller} we know that $\sigma_{ess}(H+\Vep)=\sigma_{ess}(H)=\sigma(H)$. Since $\sigma(H^e_\epsilon)$ is the disjoint union of the pure point spectrum and the essential spectrum, we resume that $\sigma_{pp}(H^e_\epsilon)\subset\rho(H)$, which means that $(H-a)^{-1}=\Phi^a$ is a bounded operator on $L^2(\Rt)^4$. By \eqref{(H-a)F=ve G},
$F=-\Phi^a \ve G$. If we multiply both sides of this last equation by $\ue$ we obtain $G=\ue F=-\ue \Phi^a \ve G$, so $-1\in \sigma_{pp}(\ue\Phi^a\ve)$ as desired.

On the contrary, assume now that there exists a nontrivial $G\in L^2(\Rt)^4$ such that $\ue \Phi^a \ve G=-G$. If we take $F=\Phi^a\ve G\in L^2(\Rt)$, we easily see that $F\not\equiv 0$ and $\Vep F=-(H-a)F$, which means that $a$ is an eigenvalue of $H^e_\epsilon$.

To conclude the  first part of the proof, it remains to show that there exists  $a \in \rho(H^e_\epsilon)$ such that $-1 \in \rho(\ue \Phi^a \ve)$. By \cite[Theorem 4.23]{thaller} we know that $\sigma_{pp}(H^e_\epsilon)$ is a finite sequence contained in $(-m,m)$, so we can chose $a\in (-m,m)\cap \rho(H^e_\epsilon)$.  Moreover, by \cite[Lemma 2]{sebaabsorption}, 
$\ue \Phi^a \ve$ is a compact operator. Then, by Fredholm's alternative,  either $-1\in \sigma_{pp}(\ue \Phi^a \ve)$ or $-1 \in\rho(\ue \Phi^a \ve)$. But we can discard the first option, otherwise $a\in\sigma_{pp}(H^e_\epsilon)$, in contradiction with $a\in \rho(H^e_\epsilon)$.

Let us now prove \eqref{Birman Shwinger}. Writing $\Vep=\ve\ue$ and using that $(H-a)^{-1}=\Phi^a$, we have
\begin{align}
(H_\epsilon^e&-a)\big(\Phi^a - \Phi^a \ve(1+\ue\Phi^a\ve)^{-1}\ue \Phi^a\big)\\
&= 1-\ve\left(1+\ue\Phi^a\ve\right)^{-1}\ue \Phi^a+\ve\ue\Phi^a
-\ve(-1+1+\ue\Phi^a\ve)\left(1+\ue\Phi^a\ve\right)^{-1}\ue \Phi^a\\
&=1-\ve\left(1+\ue\Phi^a\ve\right)^{-1}\ue \Phi^a+\ve\ue\Phi^a+\ve\left(1+\ue\Phi^a\ve\right)^{-1}\ue \Phi^a-\ve\ue\Phi^a=1,
\end{align}
as desired. This completes the proof of the proposition.
\end{proof}

The following result can be proved in the same way, we leave the details for the reader.
\begin{proposition}\label{propo 28 scalar}
Let $\ue$ and $\ve$ be as in  \eqref{eq u,v}. Then,
\begin{enumerate}[label=$(\roman*)$]
\item $a\in\rho(H^s_\epsilon)$ if and only if $-1\in\rho(\beta\ue\Phi^a\ve)$,
\item $a\in\sigma_{pp}(H^s_\epsilon)$ if and only if $-1\in \sigma_{pp}(\beta \ue\Phi^a\ve)$. Moreover, the multiplicity of $a$ as eigenvalue of $H^s_\epsilon$ coincides with the multiplicity of $-1$ as eigenvalue of $\beta\ue\Phi^a\ve$.
\end{enumerate} 
Furthermore, the following resolvent formula holds:
\begin{equation}\label{Birman Shwinger scalar}
(H^s_\epsilon-a)^{-1}=\Phi^a - \Phi^a \ve\left(\beta+\ue\Phi^a\ve\right)^{-1}\ue \Phi^a.
\end{equation}
\end{proposition}

\section{The main decomposition and the proof of Theorem \ref{Main theorem}}\label{s main deco}
Following the ideas in \cite{sebaklein,approximation}, the first key step to prove Theorem \ref{Main theorem} is to decompose 
 $(H^e_\epsilon -a)^{-1}$ and $(H^s_\epsilon -a)^{-1}$, using a scaling operator, in terms of  the operators $A_\epsilon(a)$, $B_\epsilon(a)$ and $C_\epsilon(a)$ introduced below (see Lemma \ref{lem mel}).

Let $\eta_0>0$ be some constant small enough to be fixed later on. In particular, we take $\eta_0$ so that \eqref{C^2 domain properties} holds for all $0<\epsilon\leq\eta_0$. Given $0<\epsilon\leq\eta_0$, define
\begin{align}
&\mathcal{I}_\epsilon:L^2(\Sigma\times (-\epsilon,\epsilon))^4\to L^2(\Omega_\epsilon)^4
\quad\text{by}\quad (\mathcal{I}_\epsilon f)(x_\S+t\nu(x_\S)):=f(x_\S,t),\\
&\mathcal{S}_\epsilon:L^2(\Sigma\times (-1,1))^4\to L^2(\Sigma\times (-\epsilon,\epsilon))^4
\quad\text{by}\quad  (\mathcal{S}_\epsilon g)(x_\S,t):=\frac{1}{\sqrt{\epsilon}}\,g\Big(x_\S,\frac{t}{\epsilon}\Big).
\end{align}
Thanks to the regularity of $\Sigma$, $\mathcal{I}_\epsilon$ is well-defined, bounded and invertible for all $0<\epsilon\leq\eta_0$ if $\eta_0$ is small enough. Note also that $\mathcal{S}_\epsilon$ is a unitary and invertible operator.

Let $0<\eta\leq\eta_0$, $V\in L^\infty(\R)$ with $\supp V\subset[-\eta,\eta]$ and $u,v\in L^\infty(\R)$ be the functions with support in $[-1,1]$ introduced in \eqref{eq u,v}, that is,
\begin{equation}\label{correc6}
u(t):=|\eta V (\eta t)|^{1/2}\quad\text{and}\quad v(t):=\sgn(V(\eta t))u(t).
\end{equation}
Using the notation related to \eqref{eqn:coaera}, for $0<\epsilon\leq\eta_0$  we consider the integral operators 
\begin{equation}\label{ABC espacios}
\begin{split}
&A_\epsilon(a):L^2(\Sigma\times(-1,1))^4\to L^2(\Rt)^4,\\
&B_\epsilon(a):L^2(\Sigma\times(-1,1))^4\to L^2(\Sigma\times(-1,1))^4,\\
&C_\epsilon(a):L^2(\Rt)^4\to L^2(\Sigma\times(-1,1))^4
\end{split}
\end{equation}
defined by
\begin{equation}\label{ABCepsilon}
\begin{split}
&(A_\epsilon(a)g)(x):=\int_{-1}^1\int_\Sigma\phi^a(x-y_\S - \epsilon s \nu (y_\S))v(s) \det(1-\epsilon s W(y_\S)) g(y_\S ,s)\,d\upsigma (y_\S)\,ds,\\
&(B_\epsilon (a)g)(x_\S ,t):= u(t)\int_{-1}^1\int_\S\phi^a (x_\S + \epsilon t \nu (x_\S) -y_\S -\epsilon s \nu (y_\S))v(s)\\
&\hskip200pt \times  \det(1-\epsilon s W(y_\S)) g(y_\S ,s)\,d\upsigma (y_\S)\,ds,\\
&(C_\epsilon(a)g)(x_\S,t):=u(t)\int_{\Rt}\phi^a(x_\S+\epsilon t\nu(x_\S)-y)g(y)\,dy.
\end{split}
\end{equation}

Recall that, given $F\in L^2(\Rt)^4$ and $f\in L^2(\upsigma)^4$, in \eqref{defi Phia} we defined $\Phi^a(F,f)$. However, in Section \ref{ss coupling Ve} we made the identification  $\Phi^a(\cdot)\equiv\Phi^a(\cdot,0)$, which enabled us to write $(H-a)^{-1}=\Phi^a$. 
Here, and in the sequel, we recover the initial definition for $\Phi^a$ given in \eqref{defi Phia} and we assume that $a\in\C\setminus\R$; now we must write $(H-a)^{-1}=\Phi^a(\cdot,0)$, which is a bounded operator in $L^2(\Rt)^4$.

Proceeding as in the proof of \cite[Lemma 3.2]{approximation}, one can show the following result.
\begin{lemma}\label{lem mel}
The following operator identities hold for all $0<\epsilon\leq\eta$:
\begin{equation}\label{correc2}
\begin{split}
&A_\epsilon(a)=\Phi^a(\cdot,0)\ve\,\mathcal{I}_\epsilon\,\mathcal{S}_\epsilon,\\
&B_\epsilon(a) =\mathcal{S}_\epsilon^{-1}\mathcal{I}_\epsilon^{-1} \ue \,\Phi^a(\cdot,0) \ve\,\mathcal{I}_\epsilon\,\mathcal{S}_\epsilon,\\
&C_\epsilon(a)=\mathcal{S}_\epsilon^{-1}\mathcal{I}_\epsilon^{-1} \ue\, \Phi^a(\cdot,0).
\end{split}
\end{equation}
Moreover, the following resolvent formulae hold:
\begin{align}\label{resolvent formula 2}
&(H^e_\epsilon -a)^{-1}
=(H-a)^{-1}+A_\epsilon(a)\big(1+B_\epsilon(a)\big)^{-1}C_\epsilon(a),\\\label{resolvent formula 2scalar}
&(H^s_\epsilon -a)^{-1}
=(H-a)^{-1}
+A_\epsilon(a)\big(\beta+B_\epsilon(a)\big)^{-1}C_\epsilon(a).
\end{align}
\end{lemma}

In \eqref{correc2}, $A_\epsilon(a)=\Phi^a(\cdot,0)\ve\,\mathcal{I}_\epsilon\,\mathcal{S}_\epsilon$ means that
$A_\epsilon(a)g=\Phi^a(\ve\,\mathcal{I}_\epsilon
\,\mathcal{S}_\epsilon\, g,0)$ for all $g\in L^2(\Sigma\times(-1,1))^4$, and similarly for $B_\epsilon(a)$ and $C_\epsilon(a)$. 
Since both $\mathcal{I}_\epsilon$ and $\mathcal{S}_\epsilon$ are an isometry, $V\in L^\infty(\R)$ is supported in $[-\eta,\eta]$ and $\Phi^a(\cdot,0)$ is bounded by assumption, from \eqref{correc2} we deduce that $A_\epsilon(a)$, $B_\epsilon(a)$ and $C_\epsilon(a)$ are 
well-defined and bounded, so \eqref{ABC espacios} is fully justified.
Once \eqref{correc2} is proved, the resolvent formulae \eqref{resolvent formula 2} and \eqref{resolvent formula 2scalar} follow from \eqref{Birman Shwinger} and \eqref{Birman Shwinger scalar}, respectively. We stress that, in \eqref{Birman Shwinger} and \eqref{Birman Shwinger scalar}, there is the abuse of notation in the definition of $\Phi^a$ commented before.

Lemma \ref{lem mel} connects $(H^e_\epsilon -a)^{-1}$ and $(H^s_\epsilon -a)^{-1}$ to $A_\epsilon(a)$, $B_\epsilon(a)$ and $C_\epsilon(a)$. When $\epsilon\to0$, the limit of the former ones is also connected to the limit of the latter ones. We now introduce those limit operators for $A_\epsilon(a)$, $B_\epsilon(a)$ and $C_\epsilon(a)$ when $\epsilon\to0$.
Let 
\begin{equation}\label{ABC espacios2}
\begin{split}
&A_0(a) : L^2(\Sigma\times (-1,1))^4\to L^2(\Rt)^4,\\
&B_0(a) : L^2(\Sigma\times(-1,1))^4\to L^2(\Sigma\times(-1,1))^4,\\
&B': L^2(\Sigma\times(-1,1))^4\to L^2(\Sigma\times(-1,1))^4,\\
&C_0(a):L^2(\Rt)^4\to L^2(\Sigma\times (-1,1))^4
\end{split}
\end{equation}
be the operators given by
\begin{equation}\label{limit operators defi}
\begin{split}
&(A_0(a) g)(x):= \int_{-1}^1 \int_\Sigma\phi^a(x-y_\Sigma)v(s)g(y_\Sigma,s)\,d\upsigma(y_\Sigma)\,ds,\\
&(B_0(a) g)(x_\S,t):=\lim_{\epsilon\to 0}u(t)\int_{-1}^1\int_{|x_\S-y_\S|>\epsilon}\phi^a (x_\S  -y_\S )v(s) g(y_\S ,s)\,d\upsigma (y_\S)\,ds,\\
&(B'g)(x_\S,t):=(\alpha\cdot \nu(x_\S))\,\frac{i}{2}\,u(t)\int_{-1}^1 \sgn(t-s)v(s) g(x_\S,s)\,ds,\\
&(C_0(a) g)(x_\Sigma,t):=u(t)\int_{\Rt}\phi^a(x_\Sigma-y)g(y)\,dy.
\end{split}
\end{equation}

The next theorem corresponds to the core of this article. Its proof is quite technical and is carried out in  Sections \ref{ss C}, \ref{ss B} and \ref{ss A}. 
We also postpone the proof of \eqref{ABC espacios2} to those sections, where each operator is studied in detail. Anyway, the boundedness of $B'$ is trivial. 

\begin{theorem}\label{conv AB th}
The following convergences of operators hold in the strong sense:
\begin{eqnarray}
&&A_\epsilon(a)\to A_0(a)\quad\text{when }\epsilon\to0,\label{convergence A}\\  
&&B_\epsilon(a)\to B_0(a)+B'\quad\text{when }\epsilon\to0,\label{conv B th}
\\
&&C_\epsilon(a)\to C_0(a)\quad\text{when }\epsilon\to0.\label{convergence C}
\end{eqnarray}
\end{theorem}

The proof of the following corollary is also postponed to Section \ref{s proof corol}. It combines Theorem \ref{conv AB th}, \eqref{resolvent formula 2} and \eqref{resolvent formula 2scalar}, but it requires some fine estimates developed in Sections \ref{ss C}, \ref{ss B} and \ref{ss A}. 

\begin{corollary}\label{convergence main}
There exist $\eta_0,\,\delta>0$ small enough only depending on $\S$ such that, for any $a\in\C\setminus\R$ with $|a|\leq1$, $0<\eta\leq\eta_0$ and $(\delta,\eta)$-small $V$ (see {\em Definition \ref{deltasmall}}), the following convergences of operators hold in the strong sense:
\begin{align}
&(H+\Vep-a)^{-1}\to 
(H-a)^{-1}+A_0(a)\big(1+B_0(a)+B'\big)^{-1}C_0(a)\quad\text{when }\epsilon\to0,\\
&(H+\beta\Vep-a)^{-1}\to 
(H-a)^{-1}+A_0(a)\big(\beta+B_0(a)+B'\big)^{-1}C_0(a)\quad\text{when }\epsilon\to0.
\end{align}
In particular, $(1+B_0(a)+B'\big)^{-1}$ and $(\beta+B_0(a)+B'\big)^{-1}$ are well-defined bounded operators in $L^2(\Sigma\times(-1,1))^4$.
\end{corollary}

\subsection{Proof of Theorem \ref{Main theorem}}\label{s2 ss1}
\mbox{}

Thanks to \cite[Theorem VIII.19]{reedsimon1}, to prove the theorem it is enough to show that, for some $a\in\C\setminus\R$, the following convergences of operators hold in the strong sense:
\begin{align}\label{main eq*1}
&(H+\Vep-a)^{-1}\to(H+\lambda_e\delta_\Sigma-a)^{-1}\quad\text{when }\epsilon\to0,\\\label{main eq*2}
&(H+\beta\Vep-a)^{-1}\to(H+\lambda_s\beta\delta_\Sigma-a)^{-1}\quad\text{when }\epsilon\to0.
\end{align}
Thus, from now on, we fix $a\in\C\setminus\R$ with $|a|\leq1$.

We introduce the operators \begin{equation}\widehat{V}:  L^2(\Sigma\times (-1,1))^4\to  L^2(\Sigma)^4\quad\text{and}\quad\widehat{U}: L^2(\Sigma)^4\to  L^2(\Sigma\times (-1,1))^4\end{equation} given by
\[
\widehat{V}f(x_\Sigma):=\int_{-1}^1 v(s)\,f(x_\Sigma , s) \, ds 
\quad\text{and}\quad
\widehat{U}f(x_\Sigma , t):=u(t)\,f(x_\Sigma).
\]
Observe that, by Fubini's theorem,
\begin{equation}\label{ABC_0 aa}
A_0(a) = \Phi^a(0,\cdot)\widehat{V},\qquad
B_0(a)=\widehat{U}{C^a_\upsigma}\widehat{V},\qquad
C_0(a)=\widehat{U}\Phi^a_\upsigma.
\end{equation}
Hence, from \Cref{convergence main}  and \eqref{ABC_0 aa} we deduce that, in the strong sense,
\begin{align}\label{eq final}
&(H+\Vep-a)^{-1}\to (H-a)^{-1}+\Phi^a(0,\cdot) \widehat{V}\big(1+\widehat{U}C_\upsigma^a\widehat{V}+B'\big)^{-1}
\widehat{U}\Phi^a_\upsigma\quad\text{when }\epsilon\to0,\\ \label{eq final'}
&(H+\beta\Vep-a)^{-1}\to (H-a)^{-1}+\Phi^a(0,\cdot) \widehat{V}\big(\beta+\widehat{U}C_\upsigma^a
\widehat{V}+B'\big)^{-1}\widehat{U}\Phi^a_\upsigma\quad\text{when }\epsilon\to0.
\end{align}

For convinience of notation, set 
\begin{equation}
\widetilde{\K}g(x_\S,t):=\K_V(g(x_\S,\cdot))(t)\quad\text{ for $g\in L^2(\S\times(-1,1))$,}
\end{equation}
where $\K_V$ is as in \eqref{correc3}. Then, we get
\begin{equation}
1+ B'=\mathbb{I}_4+(\alpha\cdot\nu) \widetilde{\K}\mathbb{I}_4=\left(\begin{matrix}
\mathbb{I}_2 & (\sigma \cdot\nu) \widetilde{\K}\mathbb{I}_2\\
(\sigma \cdot\nu) \widetilde{\K}\mathbb{I}_2 & \mathbb{I}_2
\end{matrix}\right).
\end{equation}
Here, $\sigma:=(\sigma_1,\sigma_2,\sigma_3)$ (see \eqref{paulimatrices}), $\mathbb{I}_4$ denotes the $4\times4$ identity matrix and $\widetilde{\K}\mathbb{I}_4$ denotes the diagonal $4\times4$ operator matrix whose nontrivial entries are $\widetilde{\K}$, and analogously for $\widetilde{\K}\mathbb{I}_2$. 
Since the operators that compose the matrix $1+B'$ commute, if we set $\K:=\widetilde{\K}\mathbb{I}_4$, we get
\begin{equation}\label{correc4}
\begin{split}
(1+B')^{-1}&=(1-\widetilde{\K}^2)^{-1}\otimes\left(\begin{matrix}
\mathbb{I}_2 & -(\sigma \cdot\nu) \widetilde{\K}\mathbb{I}_2\\
-(\sigma \cdot\nu) \widetilde{\K}\mathbb{I}_2 & \mathbb{I}_2
\end{matrix}\right)\\
&=(1-\K^2)^{-1}-(\alpha\cdot \nu) (1-\K^2)^{-1} \K.
\end{split}
\end{equation}
With this at hand, we can compute
\begin{equation}\label{final eq1}
\begin{split}
(1+\widehat{U}C^a_\upsigma \widehat{V}+B')^{-1}
&=\Big(1+(1+B')^{-1}\widehat{U}C_\upsigma^a\widehat{V}\Big)^{-1}(1+B')^{-1}\\
&=\Big(1+(1-\K^2)^{-1}
\widehat{U}C_\upsigma^a\widehat{V}-(\alpha\cdot\nu) 
(1-\K^2)^{-1}\K \widehat{U}C_\upsigma^a\widehat{V}\Big)^{-1}\\
&\hskip120pt\circ\Big((1-\K^2)^{-1}-(\alpha \cdot\nu) 
(1-\K^2)^{-1} \K\Big).
\end{split}
\end{equation}
Note that
\begin{equation}
\begin{split}
\widehat{V}\Big(1+(1-\K^2)^{-1}\widehat{U}C^a_\upsigma&
\widehat{V}-(\alpha\cdot\nu) (1-\K^2)^{-1}\K \widehat{U}C^a_\upsigma\widehat{V}\Big)\\
&=\Big(1+\widehat{V}(1-\K^2)^{-1}\widehat{U} C^a_\upsigma-(\alpha\cdot\nu) \widehat{V}(1-\K^2)^{-1}\K\widehat{U}C_\upsigma^a\Big){\widehat{V}},
\end{split}
\end{equation}
which obviously yields
\begin{equation}\label{final eq2}
\begin{split}
\widehat{V}\Big(1+(1-\K^2)^{-1}\widehat{U}C^a_\upsigma&
\widehat{V}-(\alpha\cdot\nu) (1-\K^2)^{-1}\K \widehat{U}C^a_\upsigma\widehat{V}\Big)^{-1}\\
&=\Big(1+\widehat{V}(1-\K^2)^{-1}\widehat{U} C^a_\upsigma-(\alpha\cdot\nu) \widehat{V}(1-\K^2)^{-1}\K\widehat{U}C_\upsigma^a\Big)^{-1}{\widehat{V}}.
\end{split}
\end{equation}
Besides, by the definition of $\K_V$ in \eqref{correc3}, we see that
\begin{equation}\label{final eq3}
\begin{split}
\widehat{V}(1-\K^2)^{-1}\widehat{U}&
=\Big({\int_\R\!v\,(1-\K_V^2)^{-1}u}\Big)\mathbb{I}_4,\\
\widehat{V}(1-\K^2)^{-1}\K\widehat{U}&
=\Big({\int_\R\!v\,(1-\K_V^2)^{-1}\K_V u}\Big)\mathbb{I}_4.
\end{split}
\end{equation}
From \eqref{def lambda elec} in Theorem \ref{Main theorem}, $\lambda_e=\int_\R\!v\,(1-\K_V^2)^{-1}u$. Observe also that $\int_\R\!v\,(1-\K_V^2)^{-1}\K_V u=0$. Hence, combining \eqref{final eq2} and \eqref{final eq3} we have that
\begin{equation}\label{final eq4}
\widehat{V}\Big(1+(1-\K^2)^{-1}
\widehat{U}C^a_\upsigma\widehat{V}-(\alpha\cdot\nu) (1-\K^2)^{-1}\K \widehat{U}C^a_\upsigma\widehat{V}\Big)^{-1}=(1+\lambda_e C_\sigma^a)^{-1}\widehat{V}.
\end{equation} 
Then, from \eqref{final eq1}, \eqref{final eq4} and \eqref{final eq3}, we finally get
\[
\Phi^a(0,\cdot)\widehat{V}(1+\widehat{U}C_\upsigma^a\widehat{V}+B')^{-1}\widehat{U}\Phi^a_\upsigma
= \Phi^a(0,\cdot)(1+\lambda_e C_\upsigma^a)^{-1} \lambda_e \Phi^a_\upsigma.
\]
This last identity combined with \eqref{eq final} and \eqref{resolvent H+lambda delta} yields \eqref{main eq*1}.

The proof of \eqref{main eq*2} follows the same lines. Similarly to \eqref{correc4}, 
\begin{equation}
(\beta+B')^{-1}=(1+\K^2)^{-1}\beta-(\alpha\cdot\nu)(1+\K^2)^{-1}.
\end{equation}
One can then make the computations analogous to  \eqref{final eq1}, \eqref{final eq2}, \eqref{final eq3} and \eqref{final eq4}. Since $\lambda_s=\int_\R\!v\,(1+\K_V^2)^{-1}u$, we now get
\[
\Phi^a(0,\cdot)\widehat{V}(\beta+\widehat{U}C_\upsigma^a\widehat{V}+B')^{-1}\widehat{U}\Phi^a_\upsigma
= \Phi^a(0,\cdot)(\beta+\lambda_s C_\upsigma^a)^{-1} \lambda_s \Phi^a_\upsigma.
\]
From this, \eqref{eq final'} and \eqref{resolvent H+lambda beta delta} we obtain \eqref{main eq*2}. This finishes the proof of \Cref{Main theorem}, except for the boundedness stated in \eqref{ABC espacios2}, the proof of Corollary \ref{convergence main} in Section \ref{s proof corol}, and Theorem \ref{conv AB th}, whose proof is fragmented as follows:  \eqref{convergence A} in Section \ref{ss A}, \eqref{conv B th} in Section \ref{ss B} and \eqref{convergence C} in Section \ref{ss C}.

\section{Proof of \eqref{convergence C}: $C_\epsilon (a)\to C_0(a)$ in the strong sense when $\epsilon\to0$} \label{ss C}
Recall from \eqref{ABCepsilon} and \eqref{limit operators defi} that $C_\epsilon(a)$  with $0<\epsilon\leq\eta_0$ and $C_0(a)$ are defined  by 
\begin{equation}
\begin{split}
&(C_\epsilon(a)g)(x_\S,t)=u(t)\int_{\Rt}\phi^a(x_\S+\epsilon t\nu(x_\S)-y)g(y)\,dy,\\
&(C_0(a)g)(x_\S,t)=u(t)\int_{\Rt}\phi^a(x_\S-y)g(y)\,dy.
\end{split}
\end{equation}
Let us first show that $C_\epsilon(a)$ is bounded from $L^2(\Rt)^4$ to $L^2(\Sigma\times(-1,1))^4$ with a norm uniformly bounded on $0\leq\epsilon\leq\eta_0$.
For this purpose, we write 
\begin{equation}\label{trace Sobolev 1}
(C_\epsilon(a)g)(x_\S,t)=u(t)(\phi^a*g)(x_\S+\epsilon t\nu(x_\S)),
\end{equation} 
where $\phi^a*g$ denotes the convolution of the matrix-valued function $\phi^a$ with the vector-valued function $g\in L^2(\Rt)^4$. 
Since we are assuming that $a\in\C\setminus\R$ and, in the definition of $\phi^a$, we are taking $\sqrt{m^2-a^2}$ with positive real part, the same arguments as the ones in the proof of \cite[Lemma 2.8]{amv1} (essentially Plancherel's theorem) show that 
\begin{equation}\|\phi^a*g\|_{H^1(\Rt)^4}
\leq C\|g\|_{L^2(\Rt)^4}\quad\text{for all }g\in L^2(\Rt)^4,\end{equation}
where $C>0$ only depends on $a$. Besides, thanks to the $C^2$ regularity of $\S$, if $\eta_0$ is small enough it is not hard to show that the Sobolev trace inequality from $H^1(\Rt)^4$ to $L^2(\S_{\epsilon t})^4$ holds for all $0\leq\epsilon\leq\eta_0$ and $t\in[-1,1]$ with a constant only depending on $\eta_0$ (and $\S$, of course). Combining these two facts, we obtain that 
\begin{equation}\label{trace Sobolev}
\|\phi^a*g\|_{L^2(\S_{\epsilon t})^4}
\leq C\|g\|_{L^2(\Rt)^4}\quad\text{for all $g\in L^2(\Rt)^4$, $0\leq\epsilon\leq\eta_0$ and $t\in[-1,1]$}.
\end{equation}

By Proposition \ref{weingarten map}, if $\eta_0$ is small enough there exists $C>0$ such that 
\begin{equation}\label{trace Sobolev 2}
C^{-1}\leq\det(1-\epsilon t W(P_\S x))\leq C\quad\text{for all $0<\epsilon\leq\eta_0$, $t\in(-1,1)$ and  $x\in\S_{\epsilon t}$}. 
\end{equation}
Therefore, an application of \eqref{trace Sobolev 1}, \eqref{eqn:coaera2}, \eqref{trace Sobolev 2} and \eqref{trace Sobolev} finally yields
\begin{equation}
\begin{split}
\|C_\epsilon(a)g\|^2_{L^2(\Sigma\times(-1,1))^4}
&=\int_{-1}^1\int_\S\big|u(t)(\phi^a*g)(x_\S+\epsilon t\nu(x_\S))\big|^2d\upsigma(x_\S)\,dt\\
&\leq\|u\|_{L^\infty(\R)}^2\int_{-1}^1\int_{\S_{\epsilon t}}
\big|\det(1-\epsilon t W(P_\S x))^{-1/2}(\phi^a*g)(x)\big|^2d\upsigma_{\epsilon t}(x)\,dt\\
&\leq C\|u\|_{L^\infty(\R)}^2\int_{-1}^1
\|\phi^a*g\|_{L^2(\S_{\epsilon t})^4}^2\,dt
\leq C\|u\|_{L^\infty(\R)}^2
\|g\|_{L^2(\Rt)^4}^2.
\end{split}
\end{equation}
That is, if $\eta_0$ is small enough there exists $C_1>0$ only depending  on $\eta_0$ and $a$ such that 
\begin{equation}\label{unif estimate Cepsilon}
\|C_\epsilon(a)\|_{L^2(\Rt)^4\to L^2(\Sigma\times(-1,1))^4}
\leq C_1\|u\|_{L^\infty(\R)}
\quad\text{for all $0\leq\epsilon\leq\eta_0$.}
\end{equation}
In particular, the boundedness stated in \eqref{ABC espacios2} holds for $C_0(a)$.

In order to prove the strong convergence of $C_\epsilon(a)$ to $C_0(a)$ when $\epsilon\to0$, fix $g\in L^2(\Rt)^4$. We must show that, given $\delta>0$, there exists $\epsilon_0>0$ such that 
\begin{equation}\label{case C eq0}
\|C_\epsilon(a)g-C_0(a)g\|_{L^2(\Sigma\times(-1,1))^4}
\leq\delta\quad\text{for all }0\leq\epsilon\leq\epsilon_0.
\end{equation}
For every $0<d\leq\eta_0$, using \eqref{unif estimate Cepsilon} we can estimate 
\begin{equation}\label{case C eq1}
\begin{split}
\|C_\epsilon(a)g-&C_0(a)g\|_{L^2(\Sigma\times(-1,1))^4}\\
&\leq\|C_\epsilon(a)(\chi_{\Omega_d}g)\|_{L^2(\Sigma\times(-1,1))^4}
+\|C_0(a)(\chi_{\Omega_d}g)\|_{L^2(\Sigma\times(-1,1))^4}\\
&\quad+\|(C_\epsilon(a)-C_0(a))(\chi_{\Rt\setminus\Omega_d}g)\|_{L^2(\Sigma\times(-1,1))^4}\\
&\leq 2C_1\|u\|_{L^\infty(\R)}\|\chi_{\Omega_d}g\|_{L^2(\Rt)^4}
+\|(C_\epsilon(a)-C_0(a))(\chi_{\Rt\setminus\Omega_d}g)\|_{L^2(\Sigma\times(-1,1))^4}.
\end{split}
\end{equation}
On one hand, since $g\in L^2(\Rt)^4$ and $\LL(\S)=0$ ($\LL$ denotes the Lebesgue measure in $\R^3$), we can take $d>0$ small enough so that 
\begin{equation}\label{case C eq2}
\|\chi_{\Omega_d}g\|_{L^2(\Rt)^4}\leq\frac{\delta}{4C_1\|u\|_{L^\infty(\R)}}.
\end{equation}
On the other hand, note that
\begin{equation}\label{case C eq2*}
|(x_\S+\epsilon t\nu(x_\S))-x_\S|=\epsilon |t||\nu(x_\S)|
\leq\epsilon\leq\frac{d}{2}=\frac{1}{2}\,\dt(\S,\Rt\setminus\Omega_d)
\leq\frac{1}{2}\,|x_\S-y|
\end{equation}
for all $0\leq\epsilon\leq\frac{d}{2}$, $t\in(-1,1)$, $x_\S\in\S$ and $y\in\Rt\setminus\Omega_d$.

As we said before, we are assuming that $a\in\C\setminus\R$ and, in the definition of $\phi^a$, we are taking $\sqrt{m^2-a^2}$ with positive real part, so the components of $\phi^a(x)$ decay exponentially as $|x|\to\infty$. In particular, there exist $C,r>0$ only depending on $a$ such that 
\begin{equation}\label{Horm est*}
\begin{split}
&|\partial\phi^a(x)|
\leq Ce^{-r|x|}\quad\text{for all }|x|\geq 1,\\
&|\partial\phi^a(x)|
\leq C|x|^{-3}\quad\text{for all }0<|x|<1,
\end{split}
\end{equation}
where by the left hand side in \eqref{Horm est*} we mean the absolute value of any derivative of any component of the matrix $\phi^a(x)$. Therefore, using the mean value theorem, \eqref{Horm est*} and \eqref{case C eq2*}, we see that there exists $C_{a,d}>0$ only depending on $a$ and $d$ such that
\begin{equation}
|\phi^a(x_\S+\epsilon t\nu(x_\S)-y)-\phi^a(x_\S-y)\big|
\leq C_{a,d}\,\frac{\epsilon}{|x_\S-y|^3}
\end{equation}
for all $0\leq\epsilon\leq\frac{d}{2}$, $t\in(-1,1)$, $x_\S\in\S$ and $y\in\Rt\setminus\Omega_d$. Hence, we can easily estimate
\begin{equation}
\begin{split}
|(C_\epsilon(a)-&C_0(a))(\chi_{\Rt\setminus\Omega_d}g)(x_\S,t)|\\
&\leq\|u\|_{L^\infty(\R)}\int_{\Rt\setminus\Omega_d}
\big|\phi^a(x_\S+\epsilon t\nu(x_\S)-y)-\phi^a(x_\S-y)\big||g(y)|\,dy\\
&\leq C_{a,d}\|u\|_{L^\infty(\R)}\int_{\Rt\setminus\Omega_d}
\frac{\epsilon|g(y)|}{|x_\S-y|^3}\,dy\\
&\leq C_{a,d}\,\epsilon\|u\|_{L^\infty(\R)}\Big(\int_{\Rt\setminus B_{d}(x_\S)}
\frac{dy}{|x_\S-y|^6}\Big)^{1/2}
\|g\|_{L^2(\Rt)^4}
\leq C'_{a,d}\,\epsilon\|u\|_{L^\infty(\R)}\|g\|_{L^2(\Rt)^4},
\end{split}
\end{equation}
where $C'_{a,d}>0$ only depends on $a$ and $d$. Then, 
\begin{equation}\label{case C eq3}
\|(C_\epsilon(a)-C_0(a))(\chi_{\Rt\setminus\Omega_d}g)\|_{L^2(\Sigma\times(-1,1))^4}
\leq C'_{a,d}\,\epsilon\|u\|_{L^\infty(\R)}\|g\|_{L^2(\Rt)^4}
\end{equation}
for a possibly bigger constant $C'_{a,d}>0$.

With these ingredients, the proof of \eqref{case C eq0} is straightforward. Given $\delta>0$, take $d>0$ small enough so that \eqref{case C eq2} holds. For this fixed $d$, take 
\begin{equation}
\epsilon_0=\min\bigg\{\frac{\delta}{2C'_{a,d}\|u\|_{L^\infty(\R)}\|g\|_{L^2(\Rt)^4}},\frac{d}{2}\bigg\}.
\end{equation}
Then, \eqref{case C eq0} follows from \eqref{case C eq1}, \eqref{case C eq2} and \eqref{case C eq3}. In conclusion, we have shown that
\begin{equation}\label{0001}
\lim_{\epsilon\to 0}\|(C_\epsilon(a)-C_0(a))g\|_{L^2(\Sigma\times(-1,1))^4}=0\quad\text{for all }g\in L^2(\Rt)^4,
\end{equation}
which is \eqref{convergence C}.

\section{Proof of \eqref{conv B th}: $B_\epsilon (a)\to B_0(a)+B'$ in the strong sense when $\epsilon\to0$} \label{ss B}
Recall from \eqref{ABCepsilon} and \eqref{limit operators defi} that $B_\epsilon(a)$  with $0<\epsilon\leq\eta_0$, $B_0(a)$ and $B'$ are defined  by
\begin{equation}
\begin{split}
&(B_\epsilon (a)g)(x_\S ,t)= u(t)\int_{-1}^1\int_\S\phi^a (x_\S + \epsilon t \nu (x_\S) -y_\S -\epsilon s \nu (y_\S))v(s)\\
&\hskip200pt \times  \det(1-\epsilon s W(y_\S)) g(y_\S ,s)\,d\upsigma (y_\S)\,ds,\\
&(B_0(a) g)(x_\S,t)=\lim_{\epsilon\to 0}u(t)\int_{-1}^1\int_{|x_\S-y_\S|>\epsilon}\phi^a (x_\S  -y_\S )v(s) g(y_\S ,s)\,ds\,d\upsigma (y_\S),\\
&(B'g)(x_\S,t)=(\alpha\cdot \nu(x_\S))\,\frac{i}{2}\,u(t)\int_{-1}^1 \sgn(t-s)v(s) g(x_\S,s)\,ds.
\end{split}
\end{equation}
We already know that $B_\epsilon(a)$ and $B'$ are bounded in $L^2(\Sigma\times(-1,1))^4$. Let us postpone to Section \ref{meB} the proof of the boundedness of $B_0(a)$ stated in \eqref{ABC espacios2}.
The first step to prove \eqref{conv B th} is to decompose $\phi^a $ as in {\cite[Lemma 3.2]{amv2}}, that is,
\begin{equation}\label{eqn:break phi}
\begin{split}
\phi^a(x)&=\frac{e^{-\sqrt{m^2-a^2}|x|}}{4\pi|x|}\Big(a+m\beta +\sqrt{m^2-a^2}\,i\alpha\cdot\frac{x}{|x|}\Big)\\
&\quad+\frac{e^{-\sqrt{m^2-a^2}|x|}-1}{4 \pi}\,i\alpha\cdot\frac{x}{|x|^3}+\frac{i}{4\pi}\,\alpha\cdot\frac{x}{|x|^3}
=:\omega^a_1(x)+\omega^a_2(x)+\omega_3(x).
\end{split}
\end{equation}
Then we can write
\begin{equation}\label{eqn:break phi2}
\begin{split}
&B_\epsilon (a)=B_{\epsilon,\omega_1^a}+B_{\epsilon,\omega_2^a}+B_{\epsilon,\omega_3},\\
&B_0 (a)=B_{0,\omega_1^a}+B_{0,\omega_2^a}+B_{0,\omega_3}, 
\end{split}
\end{equation}
where $B_{\epsilon,\omega_1^a}$, $B_{\epsilon,\omega_2^a}$ and $B_{\epsilon,\omega_3}$ are defined as $B_\epsilon(a)$ but replacing $\phi^a$ by $\omega_1^a$, $\omega_2^a$ and $\omega_3$, respectively, and analogously for the case of $B_0(a)$. 

For $j=1,2$, we see that $|\omega_j^a(x)|= O(|x|^{-1})$ and 
$|\partial\omega_j^a(x)|= O(|x|^{-2})|$ for $|x|\to 0$, with the understanding that $|\omega_j^a(x)|$ means the absolute value of any component of the matrix $\omega_j^a(x)$ and $|\partial\omega_j^a(x)|$ means the absolute value of any first order derivative of any component of $\omega_j^a(x)$. Therefore, the integrals defining $B_{\epsilon,\omega_j^a}$ and $B_{0,\omega_j^a}$ are of fractional type for $j=1,2$ (recall Lemma \ref{2d AD regularity}) and they are taken over bounded sets, so the strong convergence follows by standard methods.
However, one can also follow the arguments in the proof of {\cite[Lemma 3.4]{approximation}} to show, for $j=1,2$, the convergence of $B_{\epsilon,\omega_j^a}$ to $B_{0,\omega_j^a}$ in the norm sense when $\epsilon\to0$, that is,
\begin{equation}\label{0002}
\lim_{\epsilon\to 0}\|B_{\epsilon,\omega_j^a}-B_{0,\omega_j^a}\|_{L^2(\Sigma\times(-1,1))^4\to L^2(\Sigma\times(-1,1))^4}=0\quad\text{for } j=1,2.
\end{equation}
A comment is in order. Since the integrals involved in \eqref{0002} are taken over $\S\times(-1,1)$, which is bounded, the exponential decay at infinity from {\cite[Proposition A.1]{approximation}} is not necessary in the setting of \eqref{conv B th}, hence the local estimate of $|\omega_j^a(x)|$ and $|\partial\omega_j^a(x)|$ near the origin is enough to adapt the proof of {\cite[Lemma 3.4]{approximation}} to get \eqref{0002}.

Thanks to \eqref{eqn:break phi2} and \eqref{0002}, to prove \eqref{conv B th} we only need to show that 
$B_{\epsilon,\omega_3}\to B_{0,\omega_3}+B'$ in the strong sense when $\epsilon\to0$. This will be done in two main steps. First, we will show that 
\begin{equation}\label{point limit}
\lim_{\epsilon\to0}B_{\epsilon,\omega_3}g(x_\S,t)
=B_{0,\omega_3}g(x_\S,t)
+B'g(x_\S,t)\quad\text{for allmost all }(x_\S,t)\in\S\times(-1,1)
\end{equation}
and all $g\in L^\infty(\S\times(-1,1))^4$ such that 
$\sup_{|t|<1}|g(x_\S,t)-g(y_\S,t)|\leq C|x_\S-y_\S|$ for all $x_\S,\,y_\S\in\S$ and some $C>0$ which may depend on $g$. This is done in Section \ref{pointwise B}. Then, for a general $g\in L^2(\S\times(-1,1))^4$, we will estimate $|B_{\epsilon,\omega_3}g(x_\S,t)|$ in terms of some bounded maximal operators that will allow us to prove the pointwise limit \eqref{point limit} for almost every $(x_\S,t)\in\S\times(-1,1)$ and the desired strong convergence of $B_{\epsilon,\omega_3}$ to $B_{0,\omega_3}+B'$, see Section \ref{meB}.

\subsection{The pointwise limit of $B_{\epsilon,\omega_3}g(x_\S,t)\text{ when }\epsilon\to0$ for $g$ in a dense subspace of $L^2(\S\times(-1,1))^4$}\label{pointwise B}
\mbox{}

Observe that the function $u$ in front of the definitions of $B_{\epsilon,\omega_3}$, $B_{0,\omega_3}$ and $B'$ does not affect to the validity of the limit in \eqref{point limit}, so we can assume without loss of generality that $u\equiv1$ in $(-1,1)$. 

We are going to prove \eqref{point limit} by showing the pointwise limit component by component, that is, we are going to work in $L^\infty(\S\times(-1,1))$ instead of $L^\infty(\S\times(-1,1))^4$. In order to do so, we need to introduce some definitions. Set 
\begin{equation}\label{CZ kernel1}
k(x):=\frac{x}{4\pi |x|^3}\quad\text{ for $x\in\Rt\setminus\{0\}$.}
\end{equation} 
Given $t\in(-1,1)$ and $0<\epsilon\leq\eta_0$ with $\eta_0$ small enough and $f\in L^\infty(\S\times(-1,1))$ such that 
$\sup_{|t|<1}|f(x_\S,t)-f(y_\S,t)|\leq C|x_\S-y_\S|$ for all $x_\S,\,y_\S\in\S$ and some $C>0$, we define
\begin{equation}
T_t^\epsilon f(x_\Sigma):=\int_{-1}^1\int_\Sigma k (x_\Sigma+\epsilon t\nu(x_\Sigma)-y_\Sigma-\epsilon s \nu(y_\Sigma))f(y_\Sigma,s)\det(1-\epsilon sW(y_\Sigma))\,d\upsigma(y_\Sigma)\,ds.
\end{equation}
By \eqref{eqn:coaera2}, 
\begin{equation}\label{eqn:det t_eps}
\begin{split}
T_t^\epsilon f(x_\Sigma)=\int_{-1}^1\int_{\Sigma_{\epsilon s}} k (x_{\epsilon t} - y_{\epsilon s})f(P_\Sigma y_{\epsilon s},s)\,d\upsigma_{\epsilon s}(y_{\epsilon s})\,ds,
\end{split}
\end{equation}
where $x_{\epsilon t}:=x_\Sigma+\epsilon t\nu(x_\Sigma)$, $y_{\epsilon s}:=y_\Sigma+\epsilon s\nu(y_\Sigma)$ and $P_\Sigma$ is given by \eqref{P Sigma}. We  also set
\begin{equation}
\begin{split}
T_t f(x_\Sigma):=\lim_{\delta\to 0}\int_{-1}^1\!\int_{|x_\Sigma-y_\Sigma|>\delta}\!\!k(x_\Sigma-y_\Sigma)f(y_\Sigma ,s)\,d\upsigma(y_\Sigma)\,ds+\frac{\nu(x_\S)}{2}\int_{-1}^1\!\sgn(t-s) f(x_\Sigma,s)\,ds.
\end{split}
\end{equation}

We are going to prove that
\begin{equation}\label{eqn:t eps to t t}
\lim_{\epsilon\to 0} T^\epsilon_t f(x_\Sigma)=T_tf(x_\Sigma)
\end{equation}
for almost all $(x_\Sigma,t)\in\S\times(-1,1)$. Once this is proved, it is not hard to get \eqref{point limit}. Indeed, note that $k=(k_1,k_2,k_3)$ with $k_j(x):=\frac{x_j}{4\pi |x|^3}$ being the scalar components of the vector kernel $k(x)$. Thus, we can write 
\begin{equation}T_t^\epsilon f(x_\Sigma)=\big((T_t^\epsilon f(x_\Sigma))_1,(T_t^\epsilon f(x_\Sigma))_2,(T_t^\epsilon f(x_\Sigma))_3\big),\end{equation}
where each $(T_t^\epsilon f(x_\Sigma))_j$ is defined as in \eqref{eqn:det t_eps} but replacing $k$ by $k_j$. Then, \eqref{eqn:t eps to t t} holds if and only if $(T^\epsilon_t f(x_\Sigma))_j\to(T_tf(x_\Sigma))_j$ when $\epsilon\to0$ for $j=1,2,3.$ From this limits, if we let $f(y_\S ,s)$ in the definitions of $T_t^\epsilon f$ and $T_tf$ be the different componens of $v(s)g(y_\S ,s)$, we easily deduce \eqref{point limit}. Thus, we are reduced to prove \eqref{eqn:t eps to t t}.

The proof of \eqref{eqn:t eps to t t} follows the strategy of the proof of {\cite[Proposition 3.30]{mitrea}}. Set 
\begin{equation}E(x):=-\frac{1}{4\pi |x|}\quad\text{for $x\in\Rt\setminus\{0\}$,}\end{equation} the fundamental solution of the Laplace operator in $\Rt$. Note that $\nabla E=k=(k_1,k_2,k_3).$ In particular, if we set $\nu=(\nu_1,\nu_2,\nu_3)$ and $x=(x_1,x_2,x_3)$, for $x\in\Rt$ and $y\in\S$ with $x\neq y$ we can decompose
\begin{equation}\label{desc K}
\begin{split}
k_j(x&-y)=\de_{x_j} E(x-y)=|\nu(y)|^2\,\de_{x_j} E(x-y)\\
&=\sum_n \nu_n(y)^2\de_{x_j} E(x-y)+\sum_n \nu_j(y)\nu_n(y)\de_{x_n} E(x-y)-\sum_n \nu_j(y)\nu_n(y)\de_{x_n}E(x-y)\\
&=\nu_j(y)\sum_n\de_{x_n}E(x-y)\nu_n(y)+\sum_n\Big( \nu_n(y)\de_{x_j}E(x-y)-\nu_j(y)\de_{x_n}E(x-y)\Big)\nu_n(y)\\
&=\nu_j(y)\nabla_{\nu(y)}E(x-y)+\sum_n \nabla^{j,n}_{\nu(y)}E(x-y)\nu_n(y),
\end{split}
\end{equation}
where we have taken
\begin{equation}\label{defi deriva}
\begin{split}
&\nabla_{\nu(y)}E(x-y):=\sum_n \nu_n(y)\de_{x_n}E(x-y)=\nabla_{\!x} E(x-y)\cdot\nu(y),\\
&\nabla^{j,n}_{\nu(y)}E(x-y):= \nu_n(y)\de_{x_j}E(x-y)-\nu_j(y)\de_{x_n}E(x-y).
\end{split}
\end{equation}
For $j,\,n\in\{1,2,3\}$ we define
\begin{equation}\label{defi deriva2}
\begin{split}
&T^\epsilon_\nu f(x_\S,t):= \int_{-1}^1\int_{\Sigma_{\epsilon s}} \nabla_{\nu_{\epsilon s}(y_{\epsilon s})}E (x_{\epsilon t} - y_{\epsilon s})f( P_\S y_{\epsilon s},s)\,d\upsigma_{\epsilon s}(y_{\epsilon s})\,ds,\\
&T^\epsilon_{j,n} f(x_\S,t):=\int_{-1}^1\int_{\Sigma_{\epsilon s}} \nabla^{j,n}_{\nu_{\epsilon s}(y_{\epsilon s})}E (x_{\epsilon t} - y_{\epsilon s})f(P_\S y_{\epsilon s},s)\,d\upsigma_{\epsilon s}(y_{\epsilon s})\,ds,
\end{split}
\end{equation}
being $\nu_{\epsilon s}(y_{\epsilon s}):=\nu(y_\S)$ a normal vector field to $\Sigma_{\epsilon s}$. Besides, the terms $\nabla_{\nu_{\epsilon s}(y_{\epsilon s})}E (x_{\epsilon t} - y_{\epsilon s})$ and $\nabla^{j,n}_{\nu_{\epsilon s}(y_{\epsilon s})}E (x_{\epsilon t} - y_{\epsilon s})$ in \eqref{defi deriva2} are defined as in \eqref{defi deriva} with the obvious replacements. 

Given $f\in L^\infty(\S\times(-1,1))$ such that 
$\sup_{|t|<1}|f(x_\S,t)-f(y_\S,t)|\leq C|x_\S-y_\S|$ for all $x_\S,\,y_\S\in\S$ and some $C>0$, by \eqref{desc K} we see that
\begin{equation}\label{eqn: def fk}
(T_t^\epsilon f(x_\Sigma))_j=T^\epsilon_\nu h_j(x_\Sigma,t)+\sum_n T_{j,n}^\epsilon h_n(x_\Sigma,t),
\end{equation}
where $h_n(P_\S y_{\epsilon s},s):= (\nu_{\epsilon s}(y_{\epsilon s}))_n\, f(P_\Sigma y_{\epsilon s},s)$ for $n=1,2,3$. We are going to prove that
\begin{align}\label{eqn:tnu convergence}
\lim_{\epsilon\to 0} T^\epsilon_\nu h_j (x_\Sigma,t)
&= \lim_{\delta \to 0} \int_{-1}^1 \int_{|x_\Sigma-y_\Sigma|>\delta}\nabla_{\nu(y_\Sigma)}E(x_\Sigma-y_\Sigma) h_j (y_\Sigma,s)\,d\upsigma(y_\Sigma)\,ds\\
&\quad+\frac{1}{2}\int_{-1}^1\sgn(t-s)h_j(x_\Sigma,s)\,ds,\notag\\
\label{eqn:tjk convergence}
\lim_{\epsilon\to 0} T_{j,n}^\epsilon h_n(x_\Sigma,t) &=\lim_{\delta \to 0} \int_{-1}^1 \int_{|x_\Sigma-y_\Sigma|>\delta}\nabla^{j,n}_{\nu(y_\Sigma)}E(x_\Sigma-y_\Sigma) h_n (y_\Sigma,s)\,d\upsigma(y_\Sigma)\,ds
\end{align}
for $n=1,2,3$. Then, combining \eqref{eqn: def fk}, \eqref{eqn:tnu convergence} and \eqref{eqn:tjk convergence}, we obtain \eqref{eqn:t eps to t t}. Therefore, it is enough to show \eqref{eqn:tnu convergence} and \eqref{eqn:tjk convergence}.

We first deal with \eqref{eqn:tnu convergence}. Remember that $\nabla E=k$ so, given $\delta>0$, from \eqref{defi deriva} and \eqref{defi deriva2} we can split
\begin{equation}
\begin{split}
T_\nu^\epsilon h_j(x_\S,t)
&=\int_{-1}^1 \int_{|x_{\epsilon s}-y_{\epsilon s}|>\delta} k (x_{\epsilon t} - y_{\epsilon s})\cdot \nu_{\epsilon s}(y_{\epsilon s})\, h_j ( P_\S y_{\epsilon s},s)\,d\upsigma_{\epsilon s}(y_{\epsilon s})\,ds\\
&\quad +\int_{-1}^1 \int_{|x_{\epsilon s}-y_{\epsilon s}|\leq\delta}\!\! k (x_{\epsilon t} - y_{\epsilon s})\cdot \nu_{\epsilon s}(y_{\epsilon s}) \Big(h_j ( P_\S y_{\epsilon s},s) - h_j(P_\S x_{\epsilon s},s)\Big)d\upsigma_{\epsilon s}(y_{\epsilon s})\,ds\\
&\quad+\int_{-1}^1  h_j(P_\S x_{\epsilon s},s) \int_{|x_{\epsilon s}-y_{\epsilon s}|\leq\delta} k (x_{\epsilon t} - y_{\epsilon s})\cdot \nu_{\epsilon s}(y_{\epsilon s})\,d\upsigma_{\epsilon s}(y_{\epsilon s})\,ds\\
&=:\mathscr{A}_{\epsilon,\delta}+\mathscr{B}_{\epsilon,\delta}+\mathscr{C}_{\epsilon,\delta},
\end{split}
\end{equation}
and we easily see that
\begin{equation}\label{T split}
\lim_{\epsilon\to 0}T^\epsilon_{\nu}h_j(x_\S,t)=\lim_{\delta\to 0}\,\lim_{\epsilon\to 0}\big(\mathscr{A}_{\epsilon,\delta}+\mathscr{B}_{\epsilon,\delta}+\mathscr{C}_{\epsilon,\delta}\big).
\end{equation}
We study the three terms on the right hand side of \eqref{T split} separately.

For the case of $\mathscr{A}_{\epsilon,\delta}$, note that $k\in C^\infty(\R^3\setminus B_\delta(0))^3$ and it has polynomial decay at $\infty$, so 
\begin{equation}|k(x)|+|\partial k(x)|\leq C<+\infty\quad\text{for all $x\in\R^3\setminus B_\delta(0)$,}\end{equation} where $C>0$ only depends on $\delta$, and $\partial k$ denotes any first order derivative of any component of $k$. Moreover, $h_j$ is bounded on $\S\times(-1,1)$ 
and $\Sigma$ is bounded and of class $C^2$. Therefore, fixed $\delta >0$, the uniform boundedness of the integrand combined with the regularity of $k$ and $\S$ and the dominated convergence theorem yields
\begin{equation}\label{AAA}
\lim_{\epsilon\to 0}\mathscr{A}_{\epsilon,\delta}=\int_{-1}^1 \int_{|x_\S-y_\S|>\delta} k (x_{\Sigma} - y_{\Sigma})\cdot \nu(y_{\Sigma})\, h_j ( y_{\Sigma},s)\,d\upsigma(y_{\Sigma})\,ds.
\end{equation}
Then, if we let $\delta \to 0$, from \eqref{AAA} we get the first term on the right hand side of \eqref{eqn:tnu convergence}.

Recall that the function $h_j$ appearing in $\mathscr{B}_{\epsilon,\delta}$ is constructed from the one in \eqref{point limit} using $v$ (see below \eqref{eqn:t eps to t t}) and $\nu_{\epsilon s}$ (see below \eqref{eqn: def fk}). Hence $h_j\in L^\infty(\S\times(-1,1))$ and 
$\sup_{|t|<1}|h_j(x_\S,t)-h_j(y_\S,t)|\leq C|x_\S-y_\S|$ for all $x_\S,\,y_\S\in\S$ and some $C>0$. Thus, if $\eta_0$ and $\delta$ are small enough, by the mean value theorem there exists $C>0$ such that
\begin{equation}\label{eqn:estimatek}
\begin{split}
\big|k (x_{\epsilon t} - y_{\epsilon s})\cdot \nu_{\epsilon s}(y_{\epsilon s})( h_j ( P_\S y_{\epsilon s},s) -h_j(P_\S x_{\epsilon s},s))\big|
\leq C\frac{|P_\S y_{\epsilon s}-P_\S x_{\epsilon s}|}{|x_{\epsilon t}-y_{\epsilon s}|^2}
\leq \frac{C}{|y_{\epsilon s}-x_{\epsilon s}|}
\end{split}
\end{equation}
for all $0\leq\epsilon\leq\eta_0$ and $|x_{\epsilon s}-y_{\epsilon s}|\leq\delta$. In the last inequality in \eqref{eqn:estimatek} we used that $P_\S$ is Lipschitz on $\Omega_{\eta_0}$ and that $|x_{\epsilon s}-y_{\epsilon s}|\leq C|x_{\epsilon t}-y_{\epsilon s}|$ if $|x_{\epsilon s}-y_{\epsilon s}|\leq\delta$ and $\delta$ is small enough (due to the regularity of $\S$). 
From the local integrability of the right hand side of \eqref{eqn:estimatek} with respect to $\upsigma_{\epsilon s}$ (see Lemma \ref{2d AD regularity}) and standard arguments, we easily deduce the existence of $C_\delta>0$ such that
$\sup_{0\leq\epsilon\leq\eta_0}|\mathscr{B}_{\epsilon, \delta}|\leq C_\delta$ and $C_\delta\to 0$ when $\delta \to 0$, see \cite[equation (A.7)]{approximation} for a similar argument. 
Then, we can resume
\begin{equation}\label{BBB}
\Big|\lim_{\delta\to 0}\lim_{\epsilon\to 0} \mathscr{B}_{\epsilon,\delta}\Big|
\leq \lim_{\delta\to 0}\sup_{0\leq\epsilon\leq\eta_0}|\mathscr{B}_{\epsilon, \delta}|\leq\lim_{\delta\to 0}C_\delta=0.
\end{equation}

Let us finally focus on $\mathscr{C}_{\epsilon,\delta}$.
Since $k=\nabla E$, from \eqref{defi deriva} we get 
\[
\int_{|x_{\epsilon s}-y_{\epsilon s}|\leq \delta}k(x_{\epsilon t}-y_{\epsilon s})\cdot\nu_{\epsilon s}(y_{\epsilon s})\,d\upsigma_{\epsilon s}(y_{\epsilon s})=\int_{|x_{\epsilon s}-y_{\epsilon s}|\leq \delta}\nabla_{\nu_{\epsilon s}(y_{\epsilon s})}E(x_{\epsilon t}-y_{\epsilon s})\,d\upsigma_{\epsilon s}(y_{\epsilon s}).
\]
Consider the set
\[
D_\delta^\epsilon(t,s):=\begin{cases}
B_\delta(x_{\epsilon s})\setminus \overline{\Omega({\epsilon,s})} &\text{if } t\leq s,\\
B_\delta(x_{\epsilon s})\cap \Omega({\epsilon, s}) &\text{if } t>s,
\end{cases}
\]
where $\Omega({\epsilon, s})$ denotes the bounded connected component of $\Rt\setminus\S_{\epsilon s}$ that contains $\Omega$ if $s\geq 0$ and that is included in $\Omega$ if $s<0$.

\begin{figure}[!h]
\stackunder[5pt]{\includegraphics[scale=0.55]{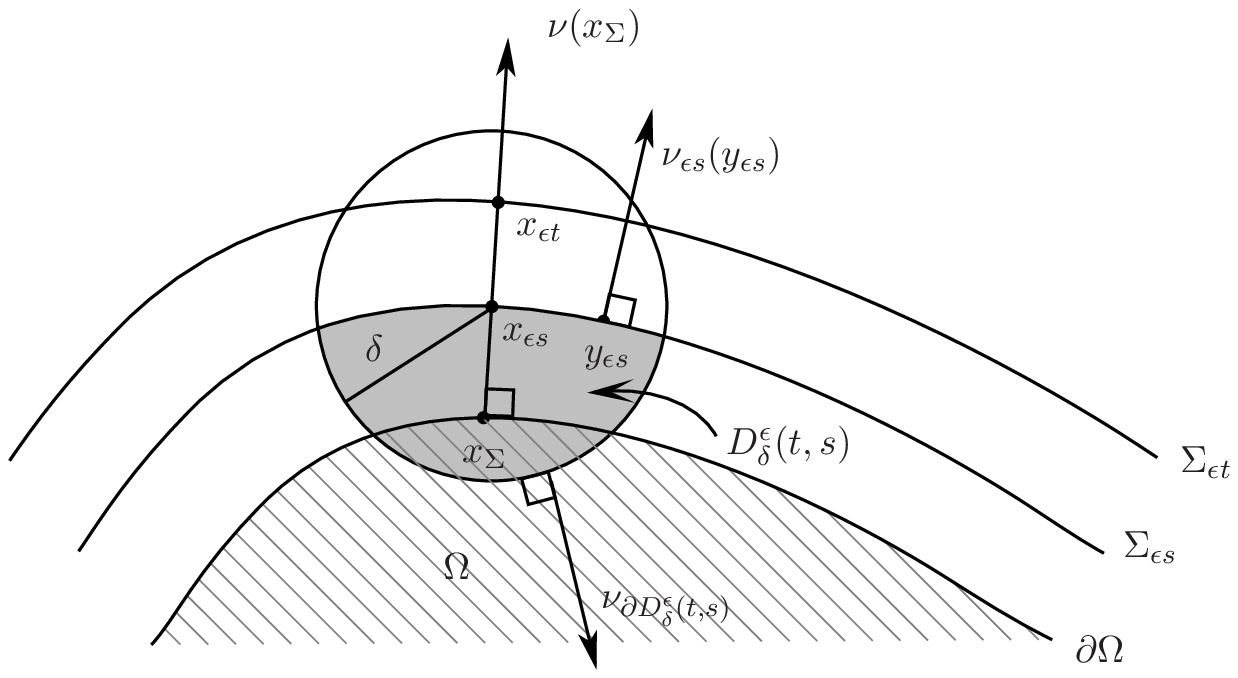}}
{$D_\delta^\epsilon(t,s)$ in the case $t>s>0$,}
\stackunder[5pt]{\includegraphics[scale=0.55]{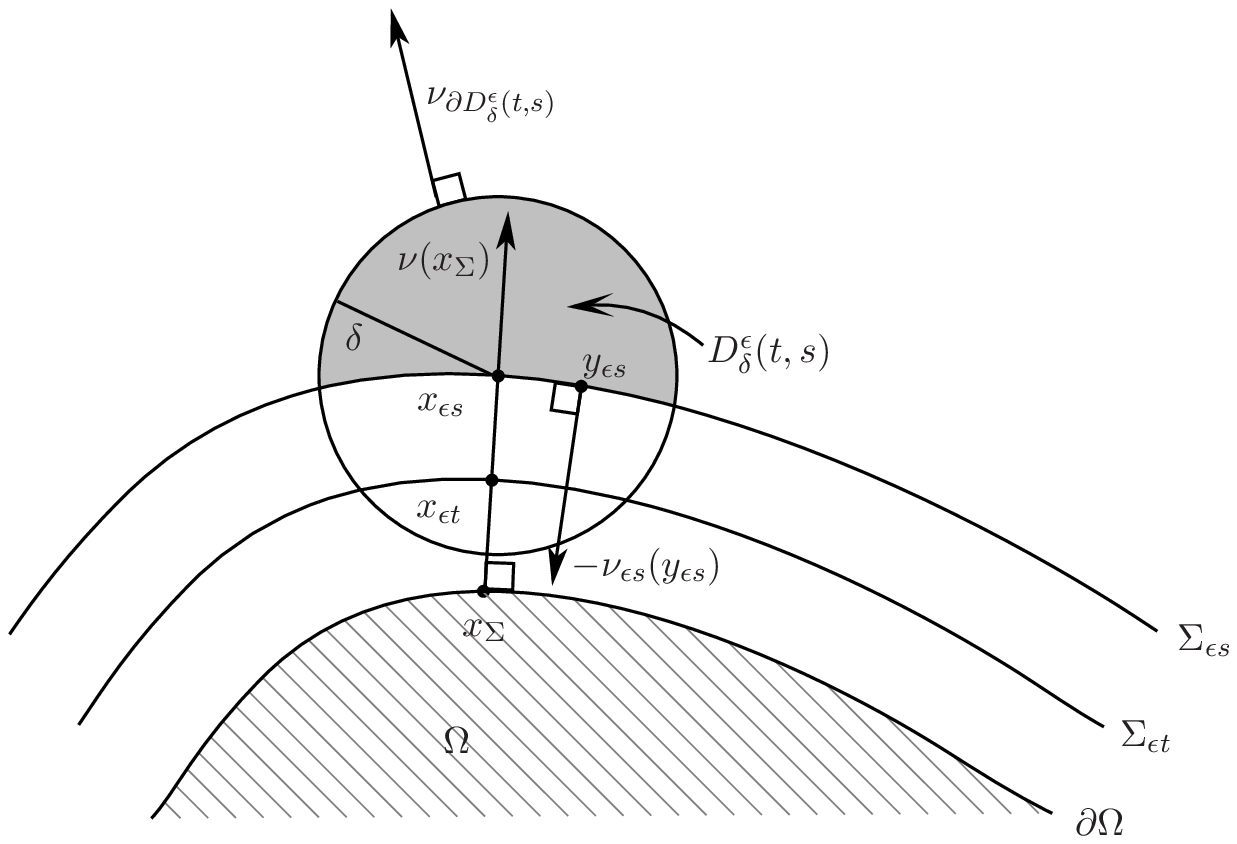}}
{$D_\delta^\epsilon(t,s)$ in the case $s>t>0$.}
\caption{The set $D_\delta^{\epsilon}(t,s)$.}
\label{figura}
\end{figure}

Set $E_x(y):=E(x-y)$ for $x,\,y\in\Rt$ with $x\neq y$. Then $\Delta E_{x_{\epsilon t}}=0$  in $D_\delta^\epsilon(t,s)$ and $\nabla  E_{x_{\epsilon t}}(y)=-\nabla E(x_{\epsilon t}-y)$.
If $\nu_{\partial D_\delta^\epsilon(t,s)}$ denotes the normal vector field on $\partial{D_\delta^\epsilon(t,s)}$ pointing outside $D_\delta^\epsilon(t,s)$, by the divergence theorem,
\begin{equation}\label{eqn:K sign*}
\begin{split}
0&=\int_{D_\delta^\epsilon(t,s)}\Delta E_{x_{\epsilon t}}(y)\,dy
=-\int_{\de D_\delta^\epsilon(t,s)}\nabla E(x_{\epsilon t}-y)\cdot\nu_{\partial D_\delta^\epsilon(t,s)}(y)\,d\mathcal{H}^2(y)\\
&=-\sgn (t-s)\int_{|x_{\epsilon s}-y_{\epsilon s}|\leq\delta}\nabla_{\nu_{\epsilon s}(y_{\epsilon s})} E(x_{\epsilon t}-y_{\epsilon s}) \,d\upsigma_{\epsilon s}(y_{\epsilon s})\\
&\quad-\int_{\{y\in\Rt:\,|x_{\epsilon s}-y|=\delta\}\cap A^\epsilon_{t,s}}\nabla E(x_{\epsilon t}-y)\cdot\frac{y-x_{\epsilon s}}{|y-x_{\epsilon s}|}\,d\mathcal{H}^2(y),
\end{split}
\end{equation}
where
\begin{equation}
\text{$A^\epsilon_{t,s}:=\Rt\setminus\overline{\Omega(\epsilon, s)}$ if $t\leq s$\qquad and \qquad$A^\epsilon_{t,s}:=\Omega(\epsilon, s)$ if $t>s$.}
\end{equation}
Remember also that  $\mathcal{H}^2$ denotes the 2-dimensional Hausdorff measure.
Since $\nabla E=k$, from \eqref{eqn:K sign*} and \eqref{defi deriva} we deduce that
\begin{equation}\label{eqn:K sign}
\begin{split}
\int_{|x_{\epsilon s}-y_{\epsilon s}|\leq\delta}k(x_{\epsilon t}-y_{\epsilon s})&\cdot\nu_{\epsilon s}(y_{\epsilon s}) \,d\upsigma_{\epsilon s}(y_{\epsilon s})\\
&=\sgn (t-s)\int_{\partial B_\delta(x_{\epsilon s})\cap A^\epsilon_{t,s}}
k(x_{\epsilon t }- y)
\cdot\frac{x_{\epsilon s}-y}{|x_{\epsilon s}-y|}\,d\mathcal{H}^2(y).
\end{split}
\end{equation}

Note that $x_{\epsilon t}\not\in D_\delta^\epsilon(t,s)$ by construction, see Figure \ref{figura}. Moreover, by the regularity of $\S$, given $\delta>0$ small enough we can find $\epsilon_0>0$ so that
$|x_{\epsilon t}-y|\geq \delta /2$ for all $0<\epsilon\leq\epsilon_0$, $s,t \in [-1,1]$ and $y\in \partial B_\delta(x_{\epsilon s})\cap A^\epsilon_{t,s}$. In particular, 
\begin{equation}\label{correc5}
|k(x_{\epsilon t}-y)|\leq C<+\infty\qquad
\text{for all }y\in \partial B_\delta(x_{\epsilon s})\cap A^\epsilon_{t,s},
\end{equation}
where $C$ only depends on $\delta$ and $\epsilon_0$.
Then,
\begin{equation}\label{lim Kepsi}
\begin{split}
\chi_{\partial B_\delta(x_{\epsilon s})\cap A^\epsilon_{t,s}}(y)\,k(x_{\epsilon t}-y&)\cdot\frac{x_{\epsilon s}-y}{|x_{\epsilon s}-y|}\,d\HH^2(y)
\\
&=\chi_{\partial B_\delta(x_{\epsilon s})\cap A^\epsilon_{t,s}}(y)\,\frac{x_{\epsilon t}-y}{4\pi|x_{\epsilon t}-y|^3}\cdot\frac{x_{\epsilon s}-y}{|x_{\epsilon s}-y|}\,d\HH^2(y)\\
&\to \frac{\chi_{\partial B_\delta(x_{\S})\cap D({t,s})}(y)}{4\pi |x_\S-y|^2}\,d\HH^2(y)\quad\text{when }\epsilon\to 0,
\end{split}
\end{equation}
where 
\begin{equation}
\text{$D(t,s):=\Rt\setminus\overline{\Omega}$ if $t\leq s$\qquad and \qquad$D(t,s):=\Omega$ if $t> s$.}
\end{equation}
The limit in \eqref{lim Kepsi} refers to weak-$*$ convergence of finite Borel measures in $\R^3$ (acting on the variable $y$).
Using \eqref{lim Kepsi}, the uniform estimate \eqref{correc5}, the boundedness of $h_j$ and the dominated convergence theorem, we see that
\begin{equation}
\begin{split}
\lim_{\epsilon\to 0}\int_{-1}^1\sgn (t-s)h_j(x_\S,&s)
\int_{\partial B_\delta(x_{\epsilon s})\cap A^\epsilon_{t,s}}
k(x_{\epsilon t }- y)\cdot\frac{x_{\epsilon s}-y}{|x_{\epsilon s}-y|}\,d\mathcal{H}^2(y)\,ds\\
&=\int_{-1}^1\sgn(t-s)h_j(x_\S,s)
\int_{\partial B_\delta(x_{\S})\cap D(t,s)}
\frac{1}{4\pi|x_\S-y|^2}\,d\mathcal{H}^2(y)\,ds\\
&=\int_{-1}^1\sgn(t-s)h_j(x_\S,s)
\frac{\mathcal{H}^2\big(\de B_\delta(x_\Sigma)\cap D(t,s)\big)}
{\mathcal{H}^2(\de B_\delta(x_\Sigma))}\,ds.
\end{split}
\end{equation}
Then, using the regularity of $\S$ and the dominated convergence theorem once again, we get
\begin{equation}\label{lim Kepsi2}
\begin{split}
\lim_{\delta\to 0}\lim_{\epsilon\to 0}\int_{-1}^1\sgn (t-s)h_j(x_\S,s)\int_{\de B_\delta(x_{\epsilon s})\cap A^\epsilon_{t,s}}
k(x_{\epsilon t }&- y)\cdot\frac{x_{\epsilon s}-y}{|x_{\epsilon s}-y|}\,d\mathcal{H}^2(y)\,ds\\
&=\frac{1}{2}\int_{-1}^1\sgn(t-s)h_j(x_\S,s)\,ds.
\end{split}
\end{equation}
By \eqref{eqn:K sign}, \eqref{lim Kepsi2} and the definition of $\mathscr{C}_{\epsilon,\delta}$ before \eqref{T split}, we get
\begin{equation}\label{CCC}
\lim_{\delta\to 0}\lim_{\epsilon\to 0}\mathscr{C}_{\epsilon,\delta}=\frac{1}{2}\int_{-1}^1\sgn(t-s)h_j(x_\Sigma,s)\,ds.
\end{equation}

The proof of \eqref{eqn:tnu convergence} is a straightforward combination of \eqref{T split}, \eqref{AAA}, \eqref{BBB} and \eqref{CCC}.

To prove \eqref{eqn:tjk convergence} we use the same approach as in \eqref{eqn:tnu convergence}, that is, we split  
\begin{equation}T_{j,n}^\epsilon h_n(x_\S,t)=:\mathscr{A}_{\epsilon,\delta}+\mathscr{B}_{\epsilon,\delta}+\mathscr{C}_{\epsilon,\delta}\end{equation} like above \eqref{T split}. The first two terms can be treated analogously and one gets the desired result, the details are left for the reader. 
To estimate $\mathscr{C}_{\epsilon,\delta}$ we use the notation introduced before. Recall that $E_{x_{\epsilon t}}$ is smooth in $\overline{D_\delta^\epsilon(t,s)}$ (assuming $t\neq s$) and $k(x_{\epsilon t}-y)=\nabla E(x_{\epsilon t}-y)=-\nabla E_{x_{\epsilon t}}(y)$. So, by the divergence theorem (see also \eqref{defi deriva}), 
\begin{equation}\label{aux1}
\begin{split}
\int_{\de D_\delta^\epsilon(t,s)}&\nabla^{j,n}_{\nu_{\de D_\delta^\epsilon(t,s)}(y)}E(x_{\epsilon t}-y)
\,d\mathcal{H}^2(y)\\
&=\int_{\de D_\delta^\epsilon(t,s)}\!\!\Big( (\nu_{\de D_\delta^\epsilon(t,s)}(y))_n\de_{x_j}E(x_{\epsilon t}-y)
-(\nu_{\de D_\delta^\epsilon(t,s)}(y))_j\de_{x_n}E(x_{\epsilon t}-y)\Big)d\mathcal{H}^2(y)\\
&=\int_{D_\delta^\epsilon(t,s)} \big(\de_{y_j}\de_{y_n} E_{x_{\epsilon t}}-\de_{y_n} \de_{y_j}E_{x_{\epsilon t}}\big)(y)\,dy=0.
\end{split}
\end{equation}
Since $\de D_\delta^\epsilon(t,s)=(B_\delta(x_{\epsilon s})\cap \Sigma_{\epsilon s})
\cup (\partial B_\delta(x_{\epsilon s})\cap A^\epsilon_{t,s})$, from \eqref{aux1} we have 
\begin{equation}\Big|\int_{|x_{\epsilon s}-y_{\epsilon s}|\leq \delta}\!\!\nabla^{j,n}_{\nu_{\epsilon s}(y_{\epsilon s})} E(x_{\epsilon st}-y_{\epsilon s})\,d\upsigma_{\epsilon s}(y_{\epsilon s})\Big|
=\Big|\int_{\partial B_\delta(x_{\epsilon s})\cap A^\epsilon_{t,s}}\!\!\nabla^{j,n}_{\nu_{\de D_\delta^\epsilon(t,s)}(y)} E(x_{\epsilon t}-y)\,d\mathcal{H}^2(y)\Big|.
\end{equation}
Observe that
\begin{equation}\label{aux2}
\begin{split}
&\chi_{\partial B_\delta(x_{\epsilon s})\cap A^\epsilon_{t,s}}(y)\,\nabla^{j,n}_{\nu_{\de D_\delta^\epsilon(t,s)}(y)} E(x_{\epsilon t}-y)\,d\mathcal{H}^2(y)\\
&\qquad=\chi_{\partial B_\delta(x_{\epsilon s})\cap A^\epsilon_{t,s}}(y)
\Big((\nu_{\de D_\delta^\epsilon(t,s)}(y))_j\de_{y_n}\!E_{x_{\epsilon t}}(y)
-(\nu_{\de D_\delta^\epsilon(t,s)}(y))_n\de_{y_j}\!E_{x_{\epsilon t}}(y)\Big)\,d\mathcal{H}^2(y)\\
&\qquad\to\chi_{\partial B_\delta(x_{\S})\cap D(t,s)}(y)
\Big(\frac{(y-x_\S)_j}{|y-x_\S|}\de_{y_n}\!E_{x_{\S}}(y)
-\frac{(y-x_\S)_n}{|y-x_\S|}\de_{y_j}\!E_{x_{\S}}(y)\Big)\,d\mathcal{H}^2(y)=0
\end{split}
\end{equation}
when $\epsilon\to0$. The limit measure in \eqref{aux2} vanishes  because its density function corresponds to a tangential derivative of $E_{x_\S}$ on $\partial B_\delta(x_{\S})$, which is a constant function on $\de B_\delta (x_\S)$. 
Therefore, arguing as in the proof of \eqref{eqn:tnu convergence} but replacing \eqref{lim Kepsi} by \eqref{aux2}, we can resume that, now, \begin{equation}\lim_{\delta\to 0}\lim_{\epsilon\to 0}\mathscr{C}_{\epsilon,\delta}=0.\end{equation}
This yields \eqref{eqn:tjk convergence} and concludes the proof of \eqref{point limit}.

\subsection{A pointwise estimate of $|B_{\epsilon,\omega_3}g(x_\S,t)|$ by maximal operators}\label{meB}
\mbox{}

We begin this section by setting
\begin{equation}\label{CZ kernel2}
\text{$k(x):=\frac{x_j}{4\pi|x|^3}$\quad for $j=1,2,3$, $x=(x_1,x_2,x_3)\in\Rt\setminus\{0\}$.}
\end{equation}
In \eqref{CZ kernel1} we already introduced a kernel $k$ which, in fact, corresponds to the vectorial version of the ones introduced in \eqref{CZ kernel2}. So, by an abuse of notation, throughout this section we mean by $k(x)$ any of the components of the kernel given in \eqref{CZ kernel1}.

Note that $k(-x)=-k(x)$ for all $x\in\Rt\setminus\{0\}$ and, besides,
there exists $C>0$ such that 
\begin{equation}\label{Horm est}
\begin{split}
&|k(x-y)|\leq \frac{C}{|x-y|^2}\quad
\text{for all $x,y\in\Rt$ such that $|x-y|>0$,}\\
&|k(z-y)-k(x-y)|\leq C\frac{|z-x|}{|x-y|^3}\quad
\text{for all $x,y,z\in\Rt$ with $0<|z-x|\leq\frac{1}{2}|x-y|$.}
\end{split}
\end{equation}

As in Section \ref{pointwise B}, we are going to work componentwise. More precisely, in order to deal with the different components of $B_{\epsilon,\omega_3}g(x_\S,t)$ for $g\in L^2(\Sigma\times(-1,1))^4$, we are going to study the following scalar version. Given $0<\epsilon\leq\eta_0$, $g\in L^2(\Sigma\times(-1,1))$ and $(x_\S,t)\in\Sigma\times(-1,1)$, define
\begin{equation}\label{witb epsilon}
\begin{split}
&\witb_\epsilon g(x_\S,t):= u(t)\int_{-1}^1\int_\S k(x_\S + \epsilon t \nu (x_\S) -y_\S -\epsilon s \nu (y_\S))\\
&\hskip170pt \times  v(s)\det(1-\epsilon s W(y_\S)) g(y_\S ,s)\,d\upsigma (y_\S)\,ds,
\end{split}
\end{equation}
where $u$ and $v$ are as in \eqref{correc6} for some $0<\eta\leq\eta_0$.
It is clear that pointwise estimates of $|\witb_\epsilon g(x_\S,t)|$ for a given $g\in L^2(\Sigma\times(-1,1))$ directly transfer to pointwise estimates of $|B_{\epsilon,\omega_3}h(x_\S,t)|$ for a given $h\in L^2(\Sigma\times(-1,1))^4$, so we are reduced to estimate
$|\witb_\epsilon g(x_\S,t)|$ for $g\in L^2(\Sigma\times(-1,1))$.

A key ingredient to find those suitable pointwise estimates  is to relate $\witb_\epsilon$  to the Hardy-Littlewood maximal operator and some  maximal singular integral operators from Calder\'on-Zygmund theory.  The
Hardy-Littlewood maximal operator is given by 
\begin{equation}\label{max hardy}
M_*f(x_\S):=\sup_{\delta>0}\frac{1}{\upsigma(B_\delta(x_\S))}\int_{B_\delta(x_\S)}|f|\,d\upsigma,
\quad\text{$M_*:L^2(\Sigma)\to L^2(\Sigma)$ bounded,}
\end{equation}
see \cite[2.19 Theorem]{mattila} for a proof of the boundedness. 
The above mentioned maximal singular integral operators are
\begin{equation}\label{max sio}
T_{*}f(x_\S):=\sup_{\delta>0}\Big|\int_{|x_\S-y_\S|>\delta}k(x_\S-y_\S)f(y_\S)\,d\upsigma(y_\S)\Big|,
\quad\text{$T_*:L^2(\Sigma)\to L^2(\Sigma)$ bounded,}
\end{equation}
see  \cite[Proposition 4 bis]{David} for a proof of the boundedness.
We also introduce some integral versions of these maximal operators to connect them to the space $L^2(\Sigma\times(-1,1))$. Set 
\begin{equation}\label{max hardy sio}
\begin{split}
&\witm_*g(x_\S):=\Big(\int_{-1}^1 M_*(g(\cdot,s))(x_\S)^2\,ds\Big)^{1/2},
\quad\text{$\witm_*:L^2(\Sigma\times(-1,1))\to L^2(\Sigma)$ bounded},\\
&\witt_{*}g(x_\S):=\int_{-1}^1 T_*(g(\cdot,s))(x_\S)\,ds,
\quad\text{$\witt_*:L^2(\Sigma\times(-1,1))\to L^2(\Sigma)$ bounded}.
\end{split}
\end{equation}
Indeed, by Fubini's theorem and \eqref{max hardy},
\begin{equation}
\begin{split}
\|\witm_*g\|_{L^2(\Sigma)}^2
&=\int_\S\int_{-1}^1 M_*(g(\cdot,s))(x_\S)^2\,ds\,d\upsigma(x_\S)
=\int_{-1}^1 \|M_*(g(\cdot,s))\|_{L^2(\S)}^2\,ds\\
&\leq C \int_{-1}^1 \|g(\cdot,s)\|_{L^2(\S)}^2\,ds
= C\|g\|_{L^2(\S\times(-1,1))}^2.
\end{split}
\end{equation}
By Cauchy-Schwarz inequality, Fubini's theorem and \eqref{max sio}, we also see that $\witt_*$ is  bounded, so \eqref{max hardy sio} is fully justified. 

Let us focus for a moment on the boundedness of $B_0(a)$ stated in \eqref{ABC espacios2}. The fact that, for $g\in L^2(\S\times(-1,1))^4$, the limit in the definition of $(B_{0}(a)g)(x_\S,t)$ exists for almost every $(x_\S,t)\in\S\times(-1,1)$ is a consequence of the decomposition (see \eqref{eqn:break phi})
\begin{equation}
\phi^a=\omega_{1}^a+\omega_{2}^a+\omega_{3},
\end{equation} 
the integrals of fractional type on bounded sets in the case of $\omega_{1}^a$ and $\omega_{2}^a$ and, for $\omega_3$,  that 
\begin{equation}\label{correc7}
\lim_{\epsilon\to0}\int_{|x_\S-y_\S|>\epsilon}k(x_\S-y_\S)f(y_\S)\,d\upsigma(y_\S)\quad\text{exists for $\upsigma$-almost every $x_\S\in\S$}
\end{equation}
if $f\in L^2(\Sigma)$
(see \cite[20.27 Theorem]{mattila} for a proof) and that 
\begin{equation}
\int_{-1}^1v(s) g(\cdot,s)\,ds
\in L^2(\Sigma)^4.
\end{equation} 
Of course, \eqref{correc7} directly applies to $B_{0,\omega_3}$ (see \eqref{eqn:break phi2} for the definition).
From the boundedness of $\witt_*$ and working component by component, we easily see that $B_{0,\omega_3}$ is bounded in $L^2(\S\times(-1,1))^4$. By the comments regarding $B_{0,\omega_1^a}$ and $B_{0,\omega_2^a}$ from the paragraph which contains \eqref{0002}, we also get that $B_{0}(a)$ is bounded in $L^2(\S\times(-1,1))^4$, which gives \eqref{ABC espacios2} in this case.

With the maximal operators at hand, we proceed to pointwise estimate 
$|\witb_\epsilon g(x_\S,t)|$ for $g\in L^2(\Sigma\times(-1,1))$.
Set 
\begin{equation}\label{witb 0bis}
g_\epsilon(y_\S,s):=v(s)\det(1-\epsilon s W(y_\S)) g(y_\S ,s).
\end{equation}
Then, since the eigenvalues of $W$ are uniformly bounded by Proposition \ref{weingarten map}, there exists $C>0$ only depending on $\eta_0$ such that
\begin{equation}\label{witb 0}
|g_\epsilon(y_\S,s)|
\leq C\|v\|_{L^\infty(\R)}|g(y_\S,s)|
\quad\text{for all }0<\epsilon\leq\eta_0,\, (y_\S,s)\in\S\times(-1,1).
\end{equation}
Besides, the regularity and boundedness of $\S$ implies the existence of $L>0$ such that 
\begin{equation}\label{witb 00}
|\nu(x_\S)-\nu(y_\S)|\leq L|x_\S-y_\S|\quad\text{for all } x_\S,y_\S\in\S.\end{equation}

We make the following splitting of 
$\witb_\epsilon g(x_\S,t)$ (see \eqref{witb epsilon} for the definition):
\begin{equation}\label{witb 1}
\begin{split}
\witb_\epsilon g(x_\S,t)\!
&=u(t)\int_{-1}^1\int_{|x_\S-y_\S|\leq4\epsilon|t-s|} k(x_\S + \epsilon t \nu (x_\S) -y_\S -\epsilon s \nu (y_\S))g_\epsilon(y_\S,s)\,d\upsigma (y_\S)\,ds\\
&\quad+u(t)\int_{-1}^1\int_{|x_\S-y_\S|>4\epsilon|t-s|} 
\Big(k(x_\S + \epsilon t \nu (x_\S) -y_\S -\epsilon s \nu (y_\S))\\
&\hskip110pt -k(x_\S + \epsilon s \nu (x_\S) -y_\S -\epsilon s \nu (y_\S))\Big)
g_\epsilon(y_\S,s)\,d\upsigma (y_\S)\,ds\\
&\quad+u(t)\int_{-1}^1\int_{|x_\S-y_\S|>4\epsilon|t-s|} 
\Big(k(x_\S + \epsilon s(\nu(x_\S)-\nu (y_\S)) -y_\S)
-k(x_\S-y_\S)\Big)\\
&\hskip268pt\times g_\epsilon(y_\S,s)\,d\upsigma (y_\S)\,ds\\
&\quad+u(t)\int_{-1}^1\int_{|x_\S-y_\S|>4\epsilon|t-s|} 
k(x_\S-y_\S)g_\epsilon(y_\S,s)\,d\upsigma (y_\S)\,ds\\
&=:\witb_{\epsilon,1}g(x_\S,t)+\witb_{\epsilon,2}g(x_\S,t)
+\witb_{\epsilon,3}g(x_\S,t)+\witb_{\epsilon,4}g(x_\S,t).
\end{split}
\end{equation}
We are going to estimate the four terms on the right hand side of \eqref{witb 1} separately.

Concerning $\witb_{\epsilon,1}g(x_\S,t)$, note that 
\begin{equation}\epsilon|t-s|=\dt(x_\S + \epsilon t \nu (x_\S),\S_{\epsilon s})
\leq|x_\S + \epsilon t \nu (x_\S) -y_\S -\epsilon s \nu (y_\S))|\end{equation}
for all $(y_\S,s)\in\Sigma\times(-1,1)$, thus
$|k(x_\S + \epsilon t \nu (x_\S) -y_\S -\epsilon s \nu (y_\S))|
\leq\frac{1}{\epsilon^2|t-s|^2}$ by \eqref{Horm est}, and then
\begin{equation}\label{witb 2}
\begin{split}
|\witb_{\epsilon,1} g(x_\S,t)|
&\leq \|u\|_{L^\infty(\R)}\int_{-1}^1\frac{1}{\epsilon^2|t-s|^2}
\int_{|x_\S-y_\S|\leq4\epsilon|t-s|}|g_\epsilon(y_\S,s)|\,d\upsigma (y_\S)\,ds\\
&\leq C\|u\|_{L^\infty(\R)}\int_{-1}^1 M_*(g_\epsilon(\cdot,s))(x_\S)\,ds
\leq C\|u\|_{L^\infty(\R)}\|v\|_{L^\infty(\R)}
\witm_*g(x_\S),
\end{split}
\end{equation}
where we used the Cauchy-Schwarz inequality and \eqref{witb 0} in the last inequality above.

For the case of $\witb_{\epsilon,2}g(x_\S,t)$, we split the integral over $\S$ on dyadic annuli as follows. Set 
\begin{equation}\label{witb 6}
\begin{split}
N:=\Big[\Big|\log_2\Big(\frac{\diam(\Omega_{\eta_0})}{\epsilon|t-s|}\Big)\Big|\Big]+1
\end{split}
\end{equation}
for $t\neq s$, where $[\,\cdot\,]$ denotes the integer part. Then, $2^N\epsilon|t-s|>\diam(\Omega_{\eta_0})$ and
\begin{equation}\label{witb 4}
\begin{split}
|\witb_{\epsilon,2}g(x_\S,t)|
&\leq\|u\|_{L^\infty(\R)}\int_{-1}^1\sum_{n=2}^N
\int_{2^{n+1}\epsilon|t-s|\geq|x_\S-y_\S|>2^n\epsilon|t-s|} 
\cdots\,\,d\upsigma (y_\S)\,ds,
\end{split}
\end{equation}
where ``$\cdots$'' means 
$
\big|k(x_\S + \epsilon t \nu (x_\S) -y_\S -\epsilon s \nu (y_\S))
-k(x_\S + \epsilon s \nu (x_\S) -y_\S -\epsilon s \nu (y_\S))\big|
|g_\epsilon(y_\S,s)|.
$
By \eqref{witb 00},
\begin{equation}
\begin{split}
(1-\eta_0 L)|x_\S-y_\S|
&\leq|x_\S-y_\S|-\eta_0|\nu(x_\S)-\nu(y_\S)|\\
&\leq|x_\S+\epsilon s\nu(x_\S)-y_\S-\epsilon s\nu(y_\S)|\\
&\leq|x_\S-y_\S|+\eta_0|\nu(x_\S)-\nu(y_\S)|\leq(1+\eta_0 L)|x_\S-y_\S|,
\end{split}
\end{equation}
thus if we take $\eta_0\leq\frac{1}{2L}$ we get
\begin{equation}\label{witb 3}
\frac{1}{2}|x_\S-y_\S|
\leq|x_\S+\epsilon s\nu(x_\S)-y_\S-\epsilon s\nu(y_\S)|
\leq2|x_\S-y_\S|.
\end{equation}
Besides, for $2^{n+1}\epsilon|t-s|\geq|x_\S-y_\S|>2^n\epsilon|t-s|$, using \eqref{witb 3} we see that
\begin{equation}\label{witb 5}
\begin{split}
|x_\S + \epsilon t \nu (x_\S) -(x_\S + \epsilon s \nu (x_\S))|
&=\epsilon|t-s|<2^{-n}|x_\S-y_\S|\\
&\leq 2^{-n+1}|x_\S+\epsilon s\nu(x_\S)-y_\S-\epsilon s\nu(y_\S)|\\
&\leq\frac{1}{2}|x_\S+\epsilon s\nu(x_\S)-y_\S-\epsilon s\nu(y_\S)|
\end{split}
\end{equation}
for all $n=2,\ldots,N$. Therefore, combining  \eqref{witb 5}, \eqref{Horm est} and \eqref{witb 3} we finally get
\begin{equation}
\begin{split}
|k(x_\S + \epsilon t \nu (x_\S) &-y_\S -\epsilon s \nu (y_\S))
-k(x_\S + \epsilon s \nu (x_\S) -y_\S -\epsilon s \nu (y_\S))\big|\\
&\leq C\frac{|x_\S + \epsilon t \nu (x_\S)-(x_\S + \epsilon s \nu (x_\S))|}
{|x_\S+\epsilon s\nu(x_\S)-y_\S-\epsilon s\nu(y_\S)|^3}
\leq\frac{C\epsilon|t-s|}{|x_\S-y_\S|^3}
<\frac{C}{2^{3n}\epsilon^2|t-s|^2}
\end{split}
\end{equation}
for all $s,t\in(-1,1)$, $0<\epsilon\leq\eta_0$, $n=2,\ldots,N$ and 
$2^{n+1}\epsilon|t-s|\geq|x_\S-y_\S|>2^n\epsilon|t-s|$. Plugging this estimate into \eqref{witb 4} we obtain
\begin{equation}\label{witb 11}
\begin{split}
|\witb_{\epsilon,2}g(x_\S,&t)|
\leq C\|u\|_{L^\infty(\R)}\int_{-1}^1\sum_{n=2}^N
\int_{2^{n+1}\epsilon|t-s|\geq|x_\S-y_\S|>2^n\epsilon|t-s|} 
\frac{|g_\epsilon(y_\S,s)|}{2^{3n}\epsilon^2|t-s|^2}\,d\upsigma (y_\S)\,ds\\
&\leq C\|u\|_{L^\infty(\R)}\int_{-1}^1\sum_{n=2}^N\frac{1}{2^n}
\int_{|x_\S-y_\S|\leq2^{n+1}\epsilon|t-s|} 
\frac{|g_\epsilon(y_\S,s)|}{(2^{n+1}\epsilon|t-s|)^2}\,d\upsigma (y_\S)\,ds\\
&\leq C\|u\|_{L^\infty(\R)}\sum_{n=2}^\infty\!\frac{1}{2^n}
\int_{-1}^1M_*(g_\epsilon(\cdot,s))(x_\S)\,ds
\leq C\|u\|_{L^\infty(\R)}\|v\|_{L^\infty(\R)}
\witm_*g(x_\S),
\end{split}
\end{equation} 
where we used the Cauchy-Schwarz inequality and \eqref{witb 0} in the last inequality above.

Let us deal now with $\witb_{\epsilon,3}g(x_\S,t)$. Since $0<\epsilon\leq\eta_0$ and $s\in(-1,1)$, if we take $\eta_0\leq\frac{1}{2L}$ as before, from \eqref{witb 00} we see that
\begin{equation}
\begin{split}
\big|\big(x_\S + \epsilon s(\nu(x_\S)-\nu (y_\S))\big)-x_\S\big|
=\epsilon |s||\nu(x_\S)-\nu (y_\S)|\leq\frac{1}{2}|x_\S-y_\S|,
\end{split}
\end{equation}
and then, by \eqref{Horm est},
\begin{equation}\label{witb 7}
\begin{split}
\big|k(x_\S + \epsilon s(\nu(x_\S)-\nu (y_\S)) -y_\S)
-k(x_\S-y_\S)\big|
\leq C\frac{\epsilon |s||\nu(x_\S)-\nu (y_\S)|}{|x_\S-y_\S|^3}
\leq \frac{C\epsilon}{|x_\S-y_\S|^2}.
\end{split}
\end{equation}
Splitting the integral which defines $\witb_{\epsilon,3}g(x_\S,t)$ into dyadic annuli as in \eqref{witb 4}, and using \eqref{witb 7}, \eqref{witb 0} and \eqref{witb 6}, we get 
\begin{equation}\label{witb 8}
\begin{split}
|\witb_{\epsilon,3}g(x_\S,t)|
&\leq C\|u\|_{L^\infty(\R)}\int_{-1}^1\sum_{n=2}^N
\epsilon\int_{2^{n+1}\epsilon|t-s|\geq|x_\S-y_\S|>2^n\epsilon|t-s|} 
\frac{|g_\epsilon(y_\S,s)|}{|x_\S-y_\S|^2}
\,d\upsigma (y_\S)\,ds\\
&\leq C\|u\|_{L^\infty(\R)}\int_{-1}^1\epsilon\sum_{n=2}^N
M_*(g_\epsilon(\cdot,s))(x_\S)\,ds\\
&\leq C\|u\|_{L^\infty(\R)}\|v\|_{L^\infty(\R)}
\int_{-1}^1\epsilon\,\Big|\log_2\Big(\frac{\diam(\Omega_{\eta_0})}{\epsilon|t-s|}\Big)\Big|
M_*(g(\cdot,s))(x_\S)\,ds.
\end{split}
\end{equation}
Note that 
\begin{equation}\epsilon\,\Big|\log_2\Big(\frac{\diam(\Omega_{\eta_0})}{\epsilon|t-s|}\Big)\Big|
\leq \epsilon\big(C+|\log_2\epsilon|+|\log_2|t-s||\big)
\leq C\big(1+|\log_2|t-s||\big)\end{equation}
for all $0<\epsilon\leq\eta_0$,
where $C>0$ only depends on $\eta_0$. Hence, from \eqref{witb 8} and Cauchy-Schwarz inequality, we obtain
\begin{equation}\label{witb 9}
\begin{split}
|\witb_{\epsilon,3}g(x_\S,t)|
&\leq C\|u\|_{L^\infty(\R)}\|v\|_{L^\infty(\R)}
\int_{-1}^1\big(1+|\log_2|t-s||\big)
M_*(g(\cdot,s))(x_\S)\,ds\\
&\leq C\|u\|_{L^\infty(\R)}\|v\|_{L^\infty(\R)}
\Big(\int_{-1}^1\big(1+|\log_2|t-s||\big)^2\,ds\Big)^{1/2}
\witm_*g(x_\S)\\
&\leq C\|u\|_{L^\infty(\R)}\|v\|_{L^\infty(\R)}
\witm_*g(x_\S),
\end{split}
\end{equation}
where we also used that $t\in(-1,1)$, so 
$\int_{-1}^1\big(1+|\log_2|t-s||\big)^2\,ds\leq
C\big(1+\int_{0}^2|\log_2r|^2\,dr\big)<+\infty$, in the last inequality above.

The term $|\witb_{\epsilon,4}g(x_\S,t)|$ can be estimated using the maximal operator $\witt_*$ as follows. Let $\lambda_1(y_\S)$ and $\lambda_2(y_\S)$ denote the eigenvalues of the Weingarten map $W(y_\S)$. By definition,
\begin{equation}
\begin{split}
g_\epsilon(y_\S,s)&=v(s)\det(1-\epsilon s W(y_\S)) g(y_\S ,s)\\
&=v(s)\big(1+\epsilon^2s^2\lambda_1(y_\S)\lambda_2(y_\S)-\epsilon s\lambda_1(y_\S)-\epsilon s\lambda_2(y_\S)\big) g(y_\S ,s).
\end{split}
\end{equation}
Therefore, the triangle inequality yields
\begin{equation}\label{witb 10}
\begin{split}
|\witb_{\epsilon,4}&g(x_\S,t)|\leq \|u\|_{L^\infty(\R)}\|v\|_{L^\infty(\R)}
\int_{-1}^1\Big(T_*(g(\cdot,s))(x_\S)
+\eta_0^2T_*(\lambda_1\lambda_2g(\cdot,s))(x_\S)\\
&\hskip165pt+\eta_0 T_*(\lambda_1g(\cdot,s))(x_\S)
+\eta_0 T_*(\lambda_2g(\cdot,s))(x_\S)\Big)\,ds\\
&\leq C\|u\|_{L^\infty(\R)}\|v\|_{L^\infty(\R)}
\big(\witt_*g(x_\S)+\witt_*(\lambda_1\lambda_2g)(x_\S)
+\witt_*(\lambda_1g)(x_\S)+\witt_*(\lambda_2g)(x_\S)\big).
\end{split}
\end{equation}

Combining \eqref{witb 1}, \eqref{witb 2}, \eqref{witb 11}, \eqref{witb 9} and \eqref{witb 10} and taking the supremum on $\epsilon$ we finally get that
\begin{equation}\label{witb 12}
\begin{split}
\sup_{0<\epsilon\leq\eta_0}|\witb_\epsilon g(x_\S,t)|
&\leq C\|u\|_{L^\infty(\R)}\|v\|_{L^\infty(\R)}
\big(\witm_*g(x_\S)+\witt_*g(x_\S)\\
&\hskip90pt+\witt_*(\lambda_1\lambda_2g)(x_\S)
+\witt_*(\lambda_1g)(x_\S)+\witt_*(\lambda_2g)(x_\S)\big),
\end{split}
\end{equation}
where $C>0$ only depends on $\eta_0$. Define 
\begin{equation}
\witb_*g(x_\S,t):=\sup_{0<\epsilon\leq\eta_0}|\witb_\epsilon g(x_\S,t)|\quad\text{ for $(x_\S,t)\in\S\times(-1,1)$.}
\end{equation} 
Then, from \eqref{witb 12}, the boundedness of $\witm_*$ and $\witt_*$ from $L^2(\Sigma\times(-1,1))$ to $L^2(\Sigma)$ (see \eqref{max hardy sio}) and the fact that 
$\|\lambda_1\|_{L^\infty(\S)}$ and 
$\|\lambda_2\|_{L^\infty(\S)}$ are finite by Proposition \ref{weingarten map}, we easily conclude that there exists $C>0$ only depending on $\eta_0$ such that
\begin{equation}\label{witb 13}
\begin{split}
\|\witb_*g\|_{L^2(\S\times(-1,1))}
\leq C\|u\|_{L^\infty(\R)}\|v\|_{L^\infty(\R)}
\|g\|_{L^2(\S\times(-1,1))}.
\end{split}
\end{equation}

\subsection{$B_{\epsilon,\omega_3}\to B_{0,\omega_3}+B'$ in the strong sense when $\epsilon\to0$ and conclusion of the proof of \eqref{conv B th}}
\label{cpB}
\mbox{}

To begin this section, we present a standard result in harmonic analysis about the existence of limit almost everywhere for a sequence of operators acting on a fixed function and its convergence in strong sense. General statements can be found in 
\cite[Theorem 2.2 and the remark below it]{duoandikoetxea} and 
\cite[Proposition 6.2]{torchinsky}, for example.
For the sake of completeness, here we present a concrete version with its proof.
\begin{lemma}\label{Calderon lemma}
Let $b\in\N$ and $(X,\mu_X)$ and $(Y,\mu_Y)$ be two Borel measure spaces. Let $\{W_\epsilon\}_{0<\epsilon\leq\eta_0}$ be a family of bounded linear operators from $L^2(\mu_X)^b$ to $L^2(\mu_Y)^b$ 
 such that, if 
\begin{equation}
W_*g(y):=\sup_{0<\epsilon\leq\eta_0}|W_\epsilon g(y)|
\quad\text{for $g\in L^2(\mu_X)^b$ and $y\in Y$,
\quad\text{then }$W_*:L^2(\mu_X)^b\to L^2(\mu_Y)$}
\end{equation}
is a bounded sublinear operator.
Suppose that for any $g\in S$, where $S\subset L^2(\mu_X)^b$ is a  dense subspace, $\lim_{\epsilon\to0}W_\epsilon g(y)$ exists for $\mu_Y$-a.e. $y\in Y$. Then, for any $g\in L^2(\mu_X)^b$, $\lim_{\epsilon\to0}W_\epsilon g(y)$ exists for $\mu_Y$-a.e. $y\in Y$ and 
\begin{equation}\label{abstract strong conv}
\lim_{\epsilon\to 0}\big\|W_\epsilon g-\lim_{\delta\to0}W_\delta g\big\|_{ L^2(\mu_Y)^b}=0.
\end{equation}
In particular, $\lim_{\epsilon\to0}W_\epsilon$ defines a bounded operator from $L^2(\mu_X)^b$ to $L^2(\mu_Y)^b$.
\end{lemma}
\begin{proof}
We start by proving that, for any $g\in L^2(\mu_X)^b$, $\lim_{\epsilon\to0}W_\epsilon g(y)$ exists for $\mu_Y$-a.e. $y\in Y$. Take $g_k\in S$ such that $\|g_k-g\|_{L^2(\mu_X)^b}\to0$ for $k\to\infty$, and fix  $\lambda>0$. Since $\lim_{\epsilon\to0}W_\epsilon g_k(y)$ exists for $\mu_Y$-a.e. $y\in Y$, Chebyshev inequality yields
\begin{equation}
\begin{split}
\mu_Y\Big(\Big\{y\in Y:&\,\Big|\limsup_{\epsilon\to0}W_\epsilon g(y)-\liminf_{\epsilon\to0}W_\epsilon g(y)\Big|>\lambda\Big\}\Big)\\
&\leq\mu_Y\Big(\Big\{y\in Y:\,\Big|\limsup_{\epsilon\to0}
W_\epsilon (g-g_k)(y)\Big|
+\Big|\liminf_{\epsilon\to0}W_\epsilon (g_k-g)(y)\Big|>\lambda\Big\}\Big)\\
&\leq\mu_Y(\{y\in Y:\,2W_* (g-g_k)(y)>\lambda\})\\
&\leq\frac{4}{\lambda^2}\,\|W_* (g-g_k)\|^2_{L^2(\mu_Y)}
\leq\frac{C}{\lambda^2}\,\|g-g_k\|^2_{L^2(\mu_X)^b}.
\end{split}
\end{equation}
Letting $k\to\infty$ we deduce that 
\begin{equation}\mu_Y\Big(\Big\{y\in Y:\,\Big|\limsup_{\epsilon\to0}W_\epsilon g(y)-\liminf_{\epsilon\to0}W_\epsilon g(y)\Big|>\lambda\Big\}\Big)=0.\end{equation} Since this holds for all $\lambda>0$, we finally get that $\lim_{\epsilon\to0}W_\epsilon g(y)$ exists $\mu_Y$-a.e.

Note that $|W_\epsilon g(y)-\lim_{\delta\to0}W_\delta g(y)|
\leq 2W_*g(y)$ and $W_*g\in L^2(\mu_Y)$. Thus, \eqref{abstract strong conv} follows by the dominated convergence theorem. The last statement in the lemma is also a consequence of the boundedness of $W_*$.
\end{proof}

Thanks to Lemma \ref{Calderon lemma} and the results in Sections \ref{pointwise B} and \ref{meB}, we are ready to conclude the proof of 
\eqref{conv B th}.
As we said before \eqref{point limit}, to obtain \eqref{conv B th} we only need to show that
$B_{\epsilon,\omega_3}\to B_{0,\omega_3}+B'$ in the strong sense when $\epsilon\to0$. From \eqref{point limit}, we know that
\begin{equation}
\lim_{\epsilon\to0}B_{\epsilon,\omega_3}g(x_\S,t)
=B_{0,\omega_3}g(x_\S,t)
+B'g(x_\S,t)\quad\text{for almost all }(x_\S,t)\in\S\times(-1,1)
\end{equation}
and all $g\in L^\infty(\S\times(-1,1))^4$ such that 
$\sup_{|t|<1}|g(x_\S,t)-g(y_\S,t)|\leq C_g|x_\S-y_\S|$ for all $x_\S,\,y_\S\in\S$ and some $C_g>0$ (it may depend on $g$). Note also that this set of functions $g$ is dense in $L^2(\S\times(-1,1))^4$. Besides, thanks to \eqref{witb 13} we see that, if $\eta_0>0$ is small enough and we set
\begin{equation}
B_{*,\omega_3}g(x_\S,t):=\sup_{0<\epsilon\leq\eta_0}
|B_{\epsilon,\omega_3}g(x_\S,t)|\quad\text{ for $(x_\S,t)\in\S\times(-1,1)$,}
\end{equation}
then there exists $C>0$ only depending on $\eta_0$ such that
\begin{equation}\label{remark eq1_}
\begin{split}
\|B_{*,\omega_3}g\|_{L^2(\S\times(-1,1))}
\leq C\|u\|_{L^\infty(\R)}\|v\|_{L^\infty(\R)}
\|g\|_{L^2(\S\times(-1,1))^4}.
\end{split}
\end{equation}
Therefore, from Lemma \ref{Calderon lemma} we get that, for any $g\in L^2(\S\times(-1,1))^4$,
the pointwise limit $\lim_{\epsilon\to0}B_{\epsilon,\omega_3}g(x_\S,t)$ exists for almost every $(x_\S,t)\in\S\times(-1,1)$. Recall also that 
$B_{0,\omega_3}+B'$ is bounded in $L^2(\S\times(-1,1))^4$ (see the comment before \eqref{witb 0bis} for $B_{0,\omega_3}$, the case of $B'$ is trivial),
so one can easily adapt the proof of Lemma \ref{Calderon lemma} to also show that, for any $g\in L^2(\S\times(-1,1))^4$,
\begin{equation}
\lim_{\epsilon\to0}B_{\epsilon,\omega_3}g(x_\S,t)
=B_{0,\omega_3}g(x_\S,t)
+B'g(x_\S,t)\quad\text{for almost all }(x_\S,t)\in\S\times(-1,1).
\end{equation}
Finally, \eqref{abstract strong conv} in Lemma \ref{Calderon lemma} yields
\begin{equation}
\lim_{\epsilon\to 0}\|(B_{\epsilon,\omega_3} -B_{0,\omega_3}-B')g\|_{L^2(\S\times(-1,1))^4}=0\quad\text{for all }g\in L^2(\S\times(-1,1))^4,
\end{equation}
which is the required strong convergence of $B_{\epsilon,\omega_3}$ to 
$B_{0,\omega_3}+B'$. This finishes the proof of \eqref{conv B th}.

\section{Proof of \eqref{convergence A}: $A_\epsilon (a)\to A_0(a)$ in the strong sense when $\epsilon\to0$}\label{ss A}
Recall from \eqref{ABCepsilon} and \eqref{limit operators defi} that $A_\epsilon(a)$  with $0<\epsilon\leq\eta_0$ and  $A_0(a)$ are defined  by
\begin{equation}
\begin{split}
&(A_\epsilon(a)g)(x)=\int_{-1}^1\int_\Sigma\phi^a(x-y_\S - \epsilon s \nu (y_\S))v(s) \det(1-\epsilon s W(y_\S)) g(y_\S ,s)\,d\upsigma (y_\S)\,ds,\\
&(A_0(a)g)(x)=\int_{-1}^1\int_\Sigma\phi^a(x-y_\S)v(s)g(y_\S ,s)\,d\upsigma (y_\S)\,ds.
\end{split}
\end{equation}
We already know that $A_\epsilon(a)$ is bounded from $L^2(\Sigma\times(-1,1))^4$ to $L^2(\Rt)^4$. To show  the boundedness of $A_0(a)$ (and conclude the proof of \eqref{ABC espacios2}) just note that, by Fubini's theorem, for every $x\in\Rt\setminus\S$ we have
\begin{equation}
\begin{split}
(A_0(a)g)(x)=\int_\Sigma\phi^a(x-y_\S)\Big(\int_{-1}^1v(s)g(y_\S ,s)\,ds\Big)d\upsigma (y_\S),
\end{split}
\end{equation}
and $\int_{-1}^1v(s)g(\cdot,s)\,ds\in L^2(\S)^4$ if $g\in L^2(\S\times(-1,1))^4$. Since $a\in\C\setminus\R$, 
\cite[Lemma 2.1]{amv1} shows that $A_0(a)$ is bounded from $L^2(\Sigma\times(-1,1))^4$ to $L^2(\Rt)^4$.

We begin the proof of \eqref{convergence A} by splitting 
\begin{equation}\label{0002A000}
A_\epsilon(a)g
=\chi_{\Rt\setminus\Omega_{\eta_0}}A_\epsilon(a)g
+\chi_{\Omega_{\eta_0}}A_\epsilon(a)g.
\end{equation}

Let us treat first the case of $\chi_{\Rt\setminus\Omega_{\eta_0}}A_\epsilon(a)$. As we said before, since $a\in \C\setminus\R$, the components of $\phi^a(x)$ decay exponentially when $|x|\to\infty$. In particular, there exist $C,r>0$ only depending on $a$ and $\eta_0$ such that 
\begin{equation}\label{A strong convergence 1}
|\phi^a(x)|,|\partial\phi^a(x)|
\leq Ce^{-r|x|}\quad\text{for all }|x|\geq \frac{\eta_0}{2},
\end{equation}
where the left hand side of \eqref{A strong convergence 1} means the absolute value of any component of the matrix $\phi^a(x)$ and of any first order derivative of it, respectively. 

Note that $\eta_0=\dt(\Rt\setminus\Omega_{\eta_0},\S)$. Hence, if $x\in\Rt\setminus\Omega_{\eta_0}$, $y_\S\in\S$, $0\leq\epsilon\leq\frac{\eta_0}{2}$ and $s\in(-1,1)$ then, for any $0\leq q\leq1$,
\begin{equation}\label{A strong convergence 2}
\begin{split}
|q(x-y_\S - \epsilon s \nu (y_\S))+(1-q)&(x-y_\S)|
=|x-y_\S-q\epsilon s \nu (y_\S)|\\
&\geq|x-y_\S|-q\epsilon |s|
\geq|x-y_\S|-\frac{\eta_0}{2}
\geq\frac{|x-y_\S|}{2}\geq\frac{\eta_0}{2}.
\end{split}
\end{equation}
Thus \eqref{A strong convergence 1} applies to 
$[x,y_\S]_q:=q(x-y_\S - \epsilon s \nu (y_\S))+(1-q)(x-y_\S)$, and a combination of the mean value theorem and \eqref{A strong convergence 2} gives
\begin{equation}\label{correc9}
\begin{split}
|\phi^a(x-y_\S - \epsilon s \nu (y_\S))-\phi^a(x-y_\S)|
\leq\epsilon\max_{0\leq q\leq1}|\partial\phi^a([x,y_\S]_q)|
\leq C\epsilon e^{-\frac{r}{2}|x-y_\S|}.
\end{split}
\end{equation}

Set 
$\widetilde{g_\epsilon}(y_\S,s):=\det(1-\epsilon s W(y_\S)) g(y_\S ,s).$
On one hand, from \eqref{correc9}, Proposition \ref{weingarten map} and Cauchy-Schwarz inequality, we get that
\begin{equation}
\begin{split}
\chi_{\Rt\setminus\Omega_{\eta_0}}(x)|
(A_\epsilon(a)&g)(x)-(A_0(a)g_\epsilon)(x)|\\
&\leq C\|v\|_{L^\infty(\R)}\chi_{\Rt\setminus\Omega_{\eta_0}}(x)\int_{-1}^1\int_\Sigma\epsilon e^{-\frac{r}{2}|x-y_\S|}|\widetilde{g_\epsilon}(y_\S ,s)|\,d\upsigma (y_\S)\,ds\\
&\leq C\epsilon\|v\|_{L^\infty(\R)}\|\widetilde{g_\epsilon}\|_{L^2(\S\times(-1,1))^4}
\chi_{\Rt\setminus\Omega_{\eta_0}}(x)
\Big(\int_\Sigma e^{-r|x-y_\S|}\,d\upsigma (y_\S)\Big)^{1/2}\\
&\leq C\epsilon\|v\|_{L^\infty(\R)}\|g\|_{L^2(\S\times(-1,1))^4}\xi(x),
\end{split}
\end{equation}
where 
\begin{equation}
\xi(x):=\chi_{\Rt\setminus\Omega_{\eta_0}}(x)
\Big(\int_\Sigma e^{-r|x-y_\S|}\,d\upsigma (y_\S)\Big)^{1/2}.
\end{equation} 
Since $\xi\in L^2(\Rt)$ because $\upsigma(\S)<+\infty$, we deduce that
\begin{equation}\label{correc10}
\begin{split}
\|\chi_{\Rt\setminus\Omega_{\eta_0}}(A_\epsilon(a)g-A_0(a)\widetilde{g_\epsilon})\|_{L^2(\Rt)^4}
\leq C\epsilon\|v\|_{L^\infty(\R)}\|g\|_{L^2(\S\times(-1,1))^4}.
\end{split}
\end{equation}
On the other hand, by Proposition \ref{weingarten map} we have that 
\begin{equation}
\begin{split}
|\widetilde{g_\epsilon}(y_\S,s)-{g}(y_\S,s)|
=\big|\!\det(1-\epsilon s W(y_\S))-1\big| |g(y_\S ,s)|
\leq C\epsilon |g(y_\S ,s)|.
\end{split}
\end{equation}
This, together with the fact that $A_0(a)$ is bounded from $L^2(\Sigma\times(-1,1))^4$ to $L^2(\Rt)^4$ (see above \eqref{0002A000}), implies that
\begin{equation}\label{correc11}
\begin{split}
\|\chi_{\Rt\setminus\Omega_{\eta_0}}A_0(a)
(\widetilde{g_\epsilon}-g)\|_{L^2(\Rt)^4}
&\leq C\|v\|_{L^\infty(\R)} 
\|\widetilde{g_\epsilon}-g\|_{L^2(\Sigma\times(-1,1))^4}\\
&\leq C\epsilon\|v\|_{L^\infty(\R)}\|g\|_{L^2(\S\times(-1,1))^4}.
\end{split}
\end{equation}

Using the triangle inequality, \eqref{correc10} and \eqref{correc11}, we finally get that
\begin{equation}\label{A exponential_}
\begin{split}
\|\chi_{\Rt\setminus\Omega_{\eta_0}}(A_\epsilon(a)-A_0(a))g\|_{L^2(\Rt)^4}
\leq C\epsilon\|v\|_{L^\infty(\R)}\|g\|_{L^2(\S\times(-1,1))^4}
\end{split}
\end{equation}
for all $0\leq\epsilon\leq\frac{\eta_0}{2}$, where $C>0$ only depends on $a$ and $\eta_0$. In particular, this implies that 
\begin{equation}\label{A exponential}
\lim_{\epsilon\to0}\|\chi_{\Rt\setminus\Omega_{\eta_0}}(A_\epsilon(a)-A_0(a))
\|_{L^2(\Sigma\times(-1,1))^4\to L^2(\Rt)^4}=0.
\end{equation}

Let us deal now with $\chi_{\Omega_{\eta_0}}A_\epsilon(a)$. Consider the decomposition of $\phi^a$ given by \eqref{eqn:break phi}.
Then, as in \eqref{eqn:break phi2}, we write
\begin{equation}\label{0002A0}
\begin{split}
&A_\epsilon (a)=A_{\epsilon,\omega_1^a}+A_{\epsilon,\omega_2^a}+A_{\epsilon,\omega_3},\\
&A_0 (a)=A_{0,\omega_1^a}+A_{0,\omega_2^a}+A_{0,\omega_3}, 
\end{split}
\end{equation}
where $A_{\epsilon,\omega_1^a}$, $A_{\epsilon,\omega_2^a}$ and $A_{\epsilon,\omega_3}$ are defined as $A_\epsilon(a)$ but replacing $\phi^a$ by $\omega_1^a$, $\omega_2^a$ and $\omega_3$, respectively, and analogously for the case of $A_0(a)$. 
For $j=1,2$, the arguments used to show \eqref{0002} in the case of
$B_{\epsilon,\omega_j^a}$ also apply to $\chi_{\Omega_{\eta_0}}A_{\epsilon,\omega_j^a}$, thus we now get 
\begin{equation}\label{0002A}
\lim_{\epsilon\to 0}\|\chi_{\Omega_{\eta_0}}(A_{\epsilon,\omega_j^a}-A_{0,\omega_j^a})\|_{L^2(\Sigma\times(-1,1))^4\to 
L^2(\Rt)^4}=0\quad\text{for } j=1,2.
\end{equation}
It only remains to show the strong convergence of $\chi_{\Omega_{\eta_0}}A_{\epsilon,\omega_3}$. This case is treated similarly to what we did in Sections \ref{pointwise B}, \ref{meB} and \ref{cpB}, as follows.

\subsection{The pointwise limit of $A_{\epsilon,\omega_3}g(x)$ when $\epsilon\to0$ for $g\in L^2(\S\times(-1,1))^4$}\label{pointwise A}
\mbox{}

This case is much more easy than the one in Section \ref{pointwise B}.
Fixed $x\in\Rt\setminus{\Sigma}$, we can always find $\delta_x,C_x>0$ small enough such that 
\begin{equation}|x-y_\Sigma-\epsilon s \nu(y_\Sigma)|\geq C_x\quad\text{for all  $y_\Sigma\in \Sigma$, $s\in (-1,1)$ and $0\leq\epsilon\leq\delta_x$.}\end{equation}
In particular, fixed $x\in\Rt\setminus{\Sigma}$, $|\omega_3(x-y_\Sigma-\epsilon s \nu(y_\Sigma))|\leq C$ uniformly on $y_\Sigma\in \Sigma$, $s\in (-1,1)$ and $0\leq\epsilon\leq\delta_x$, where $C>0$ depends on $x$. By \Cref{weingarten map} and the dominated convergence theorem, given $g\in L^2(\S\times(-1,1))^4$, we have
\begin{equation}\label{0003A}
\lim_{\epsilon\to0}A_{\epsilon,\omega_3}g(x)
=A_{0,\omega_3}g(x)\quad\text{for $\LL$-a.e. }x\in\Rt,
\end{equation}
where $\LL$ denotes the Lebesgue measure in $\R^3$.

\subsection{A pointwise estimate of $\chi_{\Omega_{\eta_0}}(x)|
A_{\epsilon,\omega_3}g(x)|$ by maximal operators}\label{meA}
\mbox{}

Given $0\leq\epsilon\leq\frac{\eta_0}{4}$, we divide the study of $\chi_{\Omega_{\eta_0}}(x)A_{\epsilon,\omega_3}g(x)$ into two different cases, i.e. $x\in\Omega_{\eta_0}\setminus\Omega_{4\epsilon}$ and $x\in\Omega_{4\epsilon}$. As we did in Section \ref{meB},  we are going to work componentwise, that is, we consider $\C$-valued functions instead of $\C^4$-valued functions.
With this in mind, for $g\in L^2(\S\times(-1,1))$ we set
\begin{equation}
\begin{split}
\wita_\epsilon g(x):=\int_{-1}^1\int_\Sigma k(x-y_\S - \epsilon s \nu (y_\S))v(s) \det(1-\epsilon s W(y_\S)) g(y_\S ,s)\,d\upsigma (y_\S)\,ds,
\end{split}
\end{equation}
where $k$ is given by \eqref{CZ kernel2}.

In what follows, we can always assume that $x\in\Rt\setminus{\Sigma}$ because $\LL(\S)=0$. In case that $x\in\Omega_{4\epsilon}$, we can write $x=x_\S+\epsilon t\nu(x_\S)$ for some $t\in(-4,4)$, and then $\wita_\epsilon g(x)$ coincides with $\witb_\epsilon g(x_\S,t)$ (see \eqref{witb epsilon}) except for the term $u(t)$. Therefore, one can carry out all the arguments involved in the estimate of $\witb_\epsilon g(x_\S,t)$ (that is, from \eqref{witb epsilon} to \eqref{witb 13}) with minor modifications to get the following result:
define 
\begin{equation}\label{A strong convergence maximal estimate 1}
\wita_*g(x_\S,t):=\sup_{0<\epsilon\leq\eta_0/4}|\wita_\epsilon g(x_\S+\epsilon t\nu(x_\S))|\quad\text{ for $(x_\S,t)\in\S\times(-4,4)$}.
\end{equation} 
Then, if $\eta_0$ is small enough, there exists $C>0$ only depending on $\eta_0$ such that
\begin{equation}\label{A strong convergence maximal estimate}
\begin{split}
\big\|\sup_{|t|< 4}\wita_*g(\cdot,t)\big\|_{L^2(\S)}
\leq C\|v\|_{L^\infty(\R)}
\|g\|_{L^2(\S\times(-1,1))}\quad\text{for all $g\in L^2(\S\times(-1,1))$.}
\end{split}
\end{equation}

For the proof of \eqref{A strong convergence maximal estimate}, a remark is in order. The fact that in the present situation $t\in(-4,4)$ instead of $t\in(-1,1)$ (as in the definition of $\witb_\epsilon g(x_\S,t)$ in \eqref{witb epsilon}) only affects the arguments used to get \eqref{witb 12} at the comment just below \eqref{witb 9}. Now one should use that $\int_0^5|\log_2 r|^2\,dr<+\infty$ to prove the estimate analogous to \eqref{witb 9} and to  derive the counterpart of \eqref{witb 12}, that is,
\begin{equation}
\begin{split}
\wita_* g(x_\S,t)
&\leq C\|v\|_{L^\infty(\R)}
\big(\witm_*g(x_\S)+\witt_*g(x_\S)
+\witt_*(\lambda_1\lambda_2g)(x_\S)
+\witt_*(\lambda_1g)(x_\S)+\witt_*(\lambda_2g)(x_\S)\big)
\end{split}
\end{equation}
for all $(x_\S,t)\in\S\times(-4,4)$, being $\lambda_1$ and $\lambda_2$ the eigenvalues of the Weingarten map. Combining this estimate (whose right hand side is independent of $t\in(-4,4)$), the boundedness of $\witm_*$ and $\witt_*$ from $L^2(\Sigma\times(-1,1))$ to $L^2(\Sigma)$ (see \eqref{max hardy sio}) and Proposition \ref{weingarten map}, we get \eqref{A strong convergence maximal estimate}. 

Finally, thanks to \eqref{A strong convergence maximal estimate 1}, \eqref{eqn:coaera}, Proposition \ref{weingarten map} and \eqref{A strong convergence maximal estimate}, for $\eta_0$  small enough we conclude 
\begin{equation}\label{cecece 1}
\begin{split}
\big\|\sup_{0\leq\epsilon\leq\eta_0/4}\chi_{\Omega_{4\epsilon}}
|\wita_\epsilon g|\big\|_{L^2(\Rt)}
&\leq\big\|\sup_{|t|<4}\wita_*g(P_\S\cdot,t)\big\|_{L^2(\Omega_{\eta_0})}\\
&\leq C\big\|\sup_{|t|< 4}\wita_*g(\cdot,t)\big\|_{L^2(\S)}
\leq C\|v\|_{L^\infty(\R)}\|g\|_{L^2(\S\times(-1,1))}.
\end{split}
\end{equation}

We now focus on $\chi_{\Omega_{\eta_0}\setminus\Omega_{4\epsilon}}
\wita_\epsilon$ for $0\leq\epsilon\leq\frac{\eta_0}{4}$.
Similarly to what we did in \eqref{witb 1}, we set 
\begin{equation}
g_\epsilon(y_\S,s):=v(s)\det(1-\epsilon s W(y_\S)) g(y_\S ,s)\qquad\text{(see \eqref{witb 0bis})}
\end{equation} 
and we split 
$\wita_\epsilon g(x)=\wita_{\epsilon,1} g(x)+\wita_{\epsilon,2} g(x)
+\wita_{\epsilon,3} g(x)+\wita_{\epsilon,4} g(x)$, where
\begin{equation}
\begin{split}
&\wita_{\epsilon,1} g(x):=\int_{-1}^1\int_{\S} 
\big(k(x-y_\S - \epsilon s \nu (y_\S))-k(x-y_\S)\big)
g_\epsilon(y_\S,s)\,d\upsigma (y_\S)\,ds,\\
&\wita_{\epsilon,2} g(x):=\int_{-1}^1\int_{|x_\S-y_\S|\leq4\dt(x,\S)}
k(x-y_\S)g_\epsilon(y_\S,s)\,d\upsigma (y_\S)\,ds,\\
&\wita_{\epsilon,3} g(x):=\int_{-1}^1\int_{|x_\S-y_\S|>4\dt(x,\S)} 
\big(k(x-y_\S)-k(x_\S-y_\S)\big)g_\epsilon(y_\S,s)\,d\upsigma (y_\S)\,ds,\\
&\wita_{\epsilon,4} g(x):=\int_{-1}^1\int_{|x_\S-y_\S|>4\dt(x,\S)} 
k(x_\S-y_\S)g_\epsilon(y_\S,s)\,d\upsigma (y_\S)\,ds.
\end{split}
\end{equation}

From now on we assume $x\in\Omega_{\eta_0}\setminus\Omega_{4\epsilon}$ and, as always, $y_\S\in\S$. Note that
\begin{equation}|(y_\S - \epsilon s \nu (y_\S))-y_\S|\leq\epsilon
\leq\frac{1}{4}\,\dt(x,\S)\leq\frac{1}{4}\,|x-y_\S|,\end{equation}
so \eqref{Horm est} gives
$|k(x-y_\S - \epsilon s \nu (y_\S))-k(x-y_\S)|\leq C\epsilon{|x-y_\S|^{-3}}.$ 
Furthermore, we have that $|x-y_\S|\geq C|x_\S-y_\S|$ for all $y_\S\in\S$ and some $C>0$ only depending on $\eta_0$.
We can split the integral on $\S$ which defines  $\wita_{\epsilon,1} g(x)$ in dyadic annuli as we did in \eqref{witb 4} (see also \eqref{witb 11}) 
to obtain
\begin{equation}\label{dedede 1}
\begin{split}
|&\wita_{\epsilon,1} g(x)|
\leq 
C\int_{-1}^1\int_{|x_\S-y_\S|<\dt(x,\S)} 
\frac{\epsilon|g_\epsilon(y_\S,s)|}{\dt(x,\S)^3}\,d\upsigma (y_\S)\,ds\\
&\qquad\qquad\quad+C\int_{-1}^1\sum_{n=0}^\infty
\int_{2^{n}\dt(x,\S)<|x_\S-y_\S|\leq2^{n+1}\dt(x,\S)} 
\frac{\epsilon|g_\epsilon(y_\S,s)|}{|x-y_\S|^3}\,d\upsigma (y_\S)\,ds\\
&\,\,\leq C\|v\|_{L^\infty(\R)}
\witm_*g(x_\S)+
C\int_{-1}^1\sum_{n=0}^\infty\frac{1}{2^n}
\int_{|x_\S-y_\S|\leq2^{n+1}\dt(x,\S)} 
\frac{|g_\epsilon(y_\S,s)|}{(2^{n}\dt(x,\S))^2}\,d\upsigma (y_\S)\,ds\\
&\,\,\leq C\|v\|_{L^\infty(\R)}
\witm_*g(x_\S)
+C\sum_{n=0}^\infty\!\frac{1}{2^n}
\int_{-1}^1M_*(g_\epsilon(\cdot,s))(x_\S)\,ds
\leq C\|v\|_{L^\infty(\R)}
\witm_*g(x_\S).
\end{split}
\end{equation} 

Using that $|k(x-y_\S)|\leq C|x-y_\S|^{-2}\leq C\dt(x,\S)^{-2}$ by \eqref{Horm est}, it is easy to show that
\begin{equation}\label{dedede 2}
|\wita_{\epsilon,2} g(x)|\leq C\|v\|_{L^\infty(\R)}\witm_*g(x_\S).
\end{equation}
Since $\dt(x,\S)=|x-x_\S|$, the same arguments as in \eqref{dedede 1} yield
\begin{equation}\label{dedede 3}
|\wita_{\epsilon,3} g(x)|\leq C\|v\|_{L^\infty(\R)}\witm_*g(x_\S).
\end{equation}
Finally, the same arguments as in \eqref{witb 10} show that
\begin{equation}\label{dedede 4}
\begin{split}
|\wita_{\epsilon,4}g(x)|
\leq C\|v\|_{L^\infty(\R)}
\big(\witt_*g(x_\S)+\witt_*(\lambda_1\lambda_2g)(x_\S)
+\witt_*(\lambda_1g)(x_\S)+\witt_*(\lambda_2g)(x_\S)\big).
\end{split}
\end{equation}
Therefore, thanks to \eqref{dedede 1}, \eqref{dedede 2}, \eqref{dedede 3} and \eqref{dedede 4} we conclude that
\begin{equation}
\begin{split}
\sup_{0\leq\epsilon\leq{\eta_0/4}}
\chi_{\Omega_{\eta_0}\setminus\Omega_{4\epsilon}}(x)
|\wita_\epsilon g(x)|
&\leq C\|v\|_{L^\infty(\R)}
\big(\witm_*g(x_\S)+\witt_*g(x_\S)\\
&\quad+\witt_*(\lambda_1\lambda_2g)(x_\S)
+\witt_*(\lambda_1g)(x_\S)+\witt_*(\lambda_2g)(x_\S)\big),
\end{split}
\end{equation}
and then, similarly to what we did in \eqref{cecece 1}, a combination of \eqref{max hardy sio} and Proposition \ref{weingarten map} gives
\begin{equation}\label{cecece 2}
\begin{split}
\big\|\sup_{0\leq\epsilon\leq{\eta_0/4}}
\chi_{\Omega_{\eta_0}\setminus\Omega_{4\epsilon}}
|\wita_\epsilon g|\big\|_{L^2(\Rt)}
&\leq C\|v\|_{L^\infty(\R)}
\|g\|_{L^2(\S\times(-1,1))}.
\end{split}
\end{equation}

Finally, combining \eqref{cecece 1} and \eqref{cecece 2} we get that, if $\eta_0>0$ is small enough, then
\begin{equation}\label{wita **}
\begin{split}
\big\|\sup_{0\leq\epsilon\leq{\eta_0/4}}
\chi_{\Omega_{\eta_0}}|\wita_\epsilon g|\big\|_{L^2(\Rt)}
&\leq C\|v\|_{L^\infty(\R)}
\|g\|_{L^2(\S\times(-1,1))},
\end{split}
\end{equation}
where $C>0$ only depends on $\eta_0$.

\subsection{$A_{\epsilon,\omega_3}\to A_{0,\omega_3}$ in the strong sense when $\epsilon\to0$ and conclusion of the proof of \eqref{convergence A}}
\mbox{}

It only remains to put all the pieces together. Despite that the proof  follows more or less the same lines as the one in Section \ref{cpB}, in this case the things are easier. Namely, now we don't need to appeal to Lemma \ref{Calderon lemma} because the dominated convergence theorem suffices (the developements in Section \ref{pointwise A} hold for all $g\in L^2(\S\times(-1,1))^4$, not only for a dense subspace like in Section \ref{pointwise B}). 

Working component by component and using  \eqref{wita **} we see that, if we set
\begin{equation}A_{*,\omega_3}g(x):=
\sup_{0\leq\epsilon\leq{\eta_0/4}}
|A_{\epsilon,\omega_3}g(x)|\quad\text{ for $x\in\Rt\setminus\S$,}\end{equation}
then there exists $C>0$ only depending on $\eta_0>0$ (being $\eta_0$ small enough) such that
\begin{equation}\label{___}
\begin{split}
\|\chi_{\Omega_{\eta_0}}A_{*,\omega_3}g\|_{L^2(\Rt)^4}
\leq C\|v\|_{L^\infty(\R)}
\|g\|_{L^2(\S\times(-1,1))^4}.
\end{split}
\end{equation}

Moreover, given $g\in L^2(\S\times(-1,1))^4$, in \eqref{0003A} we showed that $\lim_{\epsilon\to 0}A_{\epsilon,\omega_3}g(x)=A_{0,\omega_3}g(x)$ for $\LL$-a.e. $x\in\Rt$. Thus \eqref{___} and the dominated convergence theorem show that
\begin{equation}\label{0002A00}
\lim_{\epsilon\to0}\|\chi_{\Omega_{\eta_0}}(A_{\epsilon,\omega_3}-A_{0,\omega_3})g\|_{L^2(\Rt)^4}=0.
\end{equation}
Then, combining \eqref{0002A000}, \eqref{0002A0}, \eqref{A exponential}, \eqref{0002A} and \eqref{0002A00}, we conclude that
\begin{equation}
\begin{split}
\lim_{\epsilon\to0}\|(A_\epsilon(a)-A_0(a))g\|_{L^2(\Rt)^4}^2
\leq\lim_{\epsilon\to0}\Big(&
\|\chi_{\Rt\setminus\Omega_{\eta_0}}(A_\epsilon(a)-A_0(a))g\|_{L^2(\Rt)^4}^2\\
&+\|\chi_{\Omega_{\eta_0}}(A_{\epsilon,\omega_1^a}-A_{0,\omega_1^a})g\|_{L^2(\Rt)^4}^2\\
&+\|\chi_{\Omega_{\eta_0}}(A_{\epsilon,\omega_2^a}-A_{0,\omega_2^a})g\|_{L^2(\Rt)^4}^2\\
&+\|\chi_{\Omega_{\eta_0}}(A_{\epsilon,\omega_3}-A_{0,\omega_3})g\|_{L^2(\Rt)^4}^2\Big)
=0
\end{split}
\end{equation}
for all $g\in L^2(\S\times(-1,1))^4$. This is precisely \eqref{convergence A}.

\section{Proof of Corollary \ref{convergence main}}\label{s proof corol}
We first prove an auxiliary result.

\begin{lemma}\label{REM}
Let $a\in\C\setminus\R$ and $\eta_0>0$ be such that \eqref{C^2 domain properties} holds for all $0<\epsilon\leq\eta_0$. If $\eta_0$ is small enough, then for any $0<\eta\leq\eta_0$ and $V\in L^\infty(\R)$ with $\supp V\subset[-\eta,\eta]$ we have that
\begin{equation}
\begin{split}
&\|A_\epsilon(a)\|_{L^2(\Sigma\times (-1,1))^4\to L^2(\Rt)^4},\\
&\|B_\epsilon(a)\|_{L^2(\Sigma\times(-1,1))^4\to L^2(\Sigma\times(-1,1))^4},\\
&\|C_\epsilon(a)\|_{L^2(\Rt)^4\to L^2(\Sigma\times (-1,1))^4}
\end{split}
\end{equation}
are uniformly bounded for all $0\leq\epsilon\leq\eta_0$, with  bounds that only depend on $a$, $\eta_0$ and $V$. 
Furthermore, if $\eta_0$ is small enough there exists $\delta>0$ only depending on $\eta_0$ such that
\begin{equation}\label{proof colo 4}
\|B_\epsilon(a)\|_{L^2(\Sigma\times(-1,1))^4\to L^2(\Sigma\times(-1,1))^4}\leq\frac{1}{3} 
\end{equation}
for all $|a|\leq1$, $0\leq\epsilon\leq{\eta_0}$, $0<\eta\leq{\eta_0}$ and all $(\delta,\eta)$-small $V$.
\end{lemma}

\begin{proof}
The first statement in the lemma comes as a byproduct of the developements carried out in Sections \ref{ss C}, \ref{ss B} and \ref{ss A};
see \eqref{unif estimate Cepsilon} for the case of $C_\epsilon(a)$, \eqref{remark eq1_} and the paragraph which contains \eqref{0002} for $B_\epsilon(a)$, and \eqref{A exponential_}, \eqref{0002A} and \eqref{___} for $A_\epsilon(a)$. 
We shoud stress that these developements are valid for any $V\in L^\infty(\R)$ with $\supp V\subset[-\eta,\eta]$, where $0<\eta\leq{\eta_0}$, hence the $(\delta,\eta)$-small assuption on $V$ in Theorem \ref{Main theorem} is only required to prove the explicit bound in the second part of the lemma, which will yield the strong convergence of $(1+B_\epsilon(a))^{-1}$ and $(\beta+B_\epsilon(a))^{-1}$  to $(1+B_0(a)+B')^{-1}$ and $(\beta+B_0(a)+B')^{-1}$, respectively, in Corollary \ref{convergence main}.

Recall the decomposition 
\begin{equation}\label{correc8}
B_\epsilon (a)=B_{\epsilon,\omega_1^a}+B_{\epsilon,\omega_2^a}+B_{\epsilon,\omega_3}
\end{equation} 
given by   \eqref{eqn:break phi2}. Thanks to \eqref{remark eq1_}, there exists $C_0>0$ only depending on $\eta_0$ such that
\begin{equation}\label{proof colo 1}
\begin{split}
\|B_{\epsilon,\omega_3}\|_{L^2(\S\times(-1,1))^4
\to L^2(\S\times(-1,1))^4}
\leq C_0\|u\|_{L^\infty(\R)}\|v\|_{L^\infty(\R)}
\quad\text{for all }0<\epsilon\leq\eta_0.
\end{split}
\end{equation}
The comments in the paragraph which contains \eqref{0002} and an inspection of the proof of \cite[Lemma 3.4]{approximation} show that 
there also exists $C_1>0$ only depending on $\eta_0$ such that, for any $|a|\leq1$ and $j=1,2$,
\begin{equation}\label{proof colo 2}
\begin{split}
\|B_{\epsilon,\omega_j^a}\|_{L^2(\S\times(-1,1))^4
\to L^2(\S\times(-1,1))^4}
\leq C_1\|u\|_{L^\infty(\R)}\|v\|_{L^\infty(\R)}
\quad\text{for all }0<\epsilon\leq\eta_0.
\end{split}
\end{equation}
Note that the kernel defining $B_{\epsilon,\omega_2^a}$ is given by 
\begin{equation}\omega_2^a(x)=\frac{e^{-\sqrt{m^2-a^2}|x|}-1}{4 \pi}\,i\alpha\cdot\frac{x}{|x|^3},\quad\text{so 
$|\omega_2^a(x)|=O\Big(\frac{\sqrt{|m^2-a^2|}}{|x|}\Big)$
for $|x|\to0$.}\end{equation}
Therefore, the kernel is of fractional type with respect to $\upsigma$, but the estimate blows up as $|a|\to\infty$. This is the reason why we restrict ourselves to $|a|\leq1$ in \eqref{proof colo 2}, where we have a uniform bound with respect to $a$. However, for proving Theorem \ref{Main theorem}, one fixed $a\in\C\setminus\R$ suffices, say $a=i$ (see \eqref{main eq*1} and \eqref{main eq*2}).

From \eqref{correc8}, \eqref{proof colo 1} and \eqref{proof colo 2}, we derive that
\begin{equation}\label{proof colo 3}
\begin{split}
\|B_{\epsilon}(a)\|_{L^2(\S\times(-1,1))^4
\to L^2(\S\times(-1,1))^4}
\leq (C_0+2C_1)\|u\|_{L^\infty(\R)}\|v\|_{L^\infty(\R)}
\quad\text{for all }0<\epsilon\leq\eta_0.
\end{split}
\end{equation}
If $V$ is $(\delta,\eta)$-small (see Definition \ref{deltasmall}) then 
$\|V\|_{L^\infty(\R)}\leq\frac{\delta}{\eta}$, so \eqref{eq u,v} yields
\begin{equation}
\|u\|_{L^\infty(\R)}\|v\|_{L^\infty(\R)}
=\eta\|V\|_{L^\infty(\R)}\leq\delta.
\end{equation} 
Taking $\delta>0$ small enough so that 
$(C_0+2C_1)\delta\leq\frac{1}{3}$, from \eqref{proof colo 3} we finally get \eqref{proof colo 4} for all $0<\epsilon\leq\eta_0$. The case of $B_0(a)$ follows similarly, just recall the paragraph previous to \eqref{witb 0bis} taking into account that the dependence of the norm of $B_0(a)$ with respect to $\|u\|_{L^\infty(\R)}\|v\|_{L^\infty(\R)}$ is the same as in the case of $0<\epsilon\leq\eta_0$.
\end{proof}

\subsection{Proof of Corollary \ref{convergence main}}
\mbox{}

We are going to prove the corollary for $(H+\Vep-a)^{-1}$, the case of $(H+\beta\Vep-a)^{-1}$ follows by the same arguments. Let $\eta_0,\,\delta>0$ be as in Lemma \ref{REM} and take $a\in\C\setminus\R$ with $|a|\leq 1$. It is trivial to show that 
\begin{equation}\|B'\|_{L^2(\Sigma\times(-1,1))^4\to L^2(\Sigma\times(-1,1))^4}
\leq C\|u\|_{L^\infty(\R)}\|v\|_{L^\infty(\R)}\end{equation}
for some $C>0$ only depending on $\S$. 
Using \eqref{eq u,v}, we can take a smaller $\delta>0$ so that, for any  $(\delta,\eta)$-small $V$ with $0<\eta\leq\eta_0$, 
\begin{equation}
\|B'\|_{L^2(\Sigma\times(-1,1))^4\to L^2(\Sigma\times(-1,1))^4}\leq C\delta\leq\frac{1}{3}.
\end{equation}
Then, from this and \eqref{proof colo 4} in Lemma \ref{REM} (with $\epsilon=0$) we deduce that 
\begin{equation}
\begin{split}
\|(1+B_0(a)+B')g\|_{L^2(\Sigma\times(-1,1))^4}
&\geq\|g\|_{L^2(\Sigma\times(-1,1))^4}
-\|(B_0(a)+B')g\|_{L^2(\Sigma\times(-1,1))^4}\\
&\geq\frac{1}{3}\|g\|_{L^2(\Sigma\times(-1,1))^4}
\end{split}
\end{equation}
for all $g\in L^2(\Sigma\times(-1,1))^4$. Therefore, $1+B_0(a)+B'$ is invertible and 
\begin{equation}\label{opop eq1}
\|(1+B_0(a)+B')^{-1}\|_{L^2(\Sigma\times(-1,1))^4\to L^2(\Sigma\times(-1,1))^4}\leq 3.
\end{equation}
This justifies the last comment in the corollary.
Similar considerations also apply to $1+B_\epsilon(a)$, so in this case we deduce that
\begin{equation}\label{opop eq2}
\|(1+B_\epsilon(a))^{-1}\|_{L^2(\Sigma\times(-1,1))^4\to L^2(\Sigma\times(-1,1))^4}\leq \frac{3}{2}
\end{equation}
for all $0<\epsilon\leq\eta_0$.
Note also that
\begin{equation}\label{opop eq3}
\begin{split}
(1+B_\epsilon(a))^{-1}-(1&+B_0(a)+B')^{-1}\\
&=(1+B_\epsilon(a))^{-1}(B_0(a)+B'-B_\epsilon(a))
(1+B_0(a)+B')^{-1}.
\end{split}
\end{equation}

Given $g\in L^2(\Sigma\times(-1,1))^4$, set 
$f=(1+B_0(a)+B')^{-1}g\in L^2(\Sigma\times(-1,1))^4$. Then,
by \eqref{opop eq3} and \eqref{opop eq2}, we see that 
\begin{equation}\label{opop eq4}
\begin{split}
\big\|\big((1+B_\epsilon(a))^{-1}-(1+&B_0(a)+B')^{-1}\big)g\big\|_{L^2(\Sigma\times(-1,1))^4}\\
&= \|(1+B_\epsilon(a))^{-1}(B_0(a)+B'-B_\epsilon(a))
f\|_{L^2(\Sigma\times(-1,1))^4}\\
&\leq \frac{3}{2}\,\|(B_0(a)+B'-B_\epsilon(a))
f\|_{L^2(\Sigma\times(-1,1))^4}.
\end{split}
\end{equation}
By \eqref{conv B th} in Theorem \ref{conv AB th}, the right hand side of \eqref{opop eq4} converges to zero when $\epsilon\to0$. Therefore, we deduce that $(1+B_\epsilon(a))^{-1}$ converges strongly to $(1+B_0(a)+B')^{-1}$ when $\epsilon\to0$. Since the composition of strongly convergent operators is strongly convergent, using \eqref{resolvent formula 2} and Theorem \ref{conv AB th}, we finally obtain the desired strong convergence 
\begin{equation}(H+\Vep-a)^{-1}\to 
(H-a)^{-1}+A_0(a)\big(1+B_0(a)+B'\big)^{-1}C_0(a)\quad\text{when }\epsilon\to0.\end{equation} Corollary \ref{convergence main} is finally proved.

\printbibliography
\end{document}